\documentclass[a4paper,11pt,reqno]{amsart}
\usepackage[english]{babel}
\usepackage[T1]{fontenc} 
\usepackage[utf8]{inputenc}
\allowdisplaybreaks  % permite cortes de página en fórmulas

\usepackage{xcolor}

\usepackage[all]{xy}
\usepackage{hyperref}
\usepackage{amsthm}
\usepackage{amssymb}

\usepackage{todonotes}
\numberwithin{equation}{section}

\newtheorem{theorem}[equation]{Theorem}
\newtheorem{corollary}[equation]{Corollary}
\newtheorem{lemma}[equation]{Lemma}
\newtheorem{proposition}[equation]{Proposition}
\newtheorem{assumption}[equation]{Assumptions}

\theoremstyle{definition}
\newtheorem{definition}[equation]{Definition}
\newtheorem{example}[equation]{Example}

\theoremstyle{remark}
\newtheorem{remark}[equation]{Remark}
\newcommand{\sslash}{ /\!\!/ }
\newcommand{\GL}{\operatorname{Gl}}
\newcommand{\SL}{\operatorname{Sl}}
\newcommand{\GLm}{\operatorname{Gl}_{m}}
\newcommand{\SLm}{\operatorname{Sl}_{m}}
\renewcommand{\P}{\mathbb{P}}

\renewcommand\O{\mathcal{O}}

\DeclareMathOperator{\Proj}{Proj}
\DeclareMathOperator{\rk}{rk}
\DeclareMathOperator{\Hom}{Hom}

\DeclareMathOperator{\Spec}{Spec}
\DeclareMathOperator{\Aut}{Aut}
\DeclareMathOperator{\End}{End}
\DeclareMathOperator{\HOM}{{\mathbb{H}om}} 
\DeclareMathOperator{\ISOM}{{\mathbb{I}som}} 

\DeclareMathOperator{\Alg}{Alg}
\DeclareMathOperator{\pAlg}{pAlg}
\DeclareMathOperator{\gAlg}{gAlg}
\DeclareMathOperator{\Mod}{Mod}
\DeclareMathOperator{\Id}{Id}

\newcommand{\Z}{\mathbb{Z}}

\newcommand{\Sim}{\operatorname{S}}

\begin{document}
\emergencystretch 3em

\title[On the subalgebra of invariant elements]{On the subalgebra of invariant elements: finiteness and immersions}

\author{J. Mart\'in Ovejero}
\address{Departamento de Matem\'aticas, Universidad de Salamanca, Spain}
%\curraddr{Department of Mathematics, University of Le\'on, Spain}
\email{lemurx@usal.es}
\thanks{The authors are supported by grant PGC2018-099599-B-I00 of MINECO}

\author{A. L. Mu\~noz Casta\~neda}
\address{Departamento de Matem\'aticas, Universidad de Le\'on, Spain}
\email{amunc@unileon.es}

\author{F. J. Plaza Mart\'in}
\address{Departamento de Matem\'aticas, Universidad de Salamanca, Spain}
\email{fplaza@usal.es}


\date{}

%\dedicatory{This paper is dedicated to ......}


%%%%%%%%%%%%%%%%%%%%%%%%%%%%%%%%%%%%%%%%%%%%%%%%%%%%%%%%%%%%%%%%%%%%
%%%%%%%%%%%%%%%%%%%%%%%%%%%%%%%%%%%%%%%%%%%%%%%%%%%%%%%%%%%%%%%%%%%%
%%%%%%%%%%%%%%%%%%%%%%%%%%%%%%%%%%%%%%%%%%%%%%%%%%%%%%%%%%%%%%%%%%%%

\begin{abstract}

Let $R$ be an algebra over a ring $\Bbbk$, $T$ an $R$-algebra, $M$ a finitely generated projective $R$-module, and $N$ a $T$-module. Let $G$ be a linearly reductive group scheme over $\Bbbk$ equipped with a representation $\rho:\underline{G}_{R}\rightarrow \underline{\Aut}_{\Mod(R)}(M)$. For the graded $T$-algebra $A$, defined as
$
A := \left( \Sim_{T}^{\bullet} (M^{\vee} \otimes_{R}N )\right)^{G},
$
we determine the conditions under which the graded $T$-algebra $A$ is finitely generated, finitely presented, or flat. Furthermore, we establish the conditions under which a closed embedding of $\Proj A$ into a projective space exists. Since we do not impose any Noetherian hypotheses, our results generalize those in the literature, providing new powerful tools regarding moduli problems.
\end{abstract}

\maketitle

\tableofcontents{}

%%%%%%%%%%%%%%%%%%%%%%%%%%%%%%%%%%%%%%%%%%%%%%%%%%%%%%%%%%%%%%%%%%%%
%%%%%%%%%%%%%%%%%%%%%%%%%%%%%%%%%%%%%%%%%%%%%%%%%%%%%%%%%%%%%%%%%%%%
%%%%%%%%%%%%%%%%%%%%%%%%%%%%%%%%%%%%%%%%%%%%%%%%%%%%%%%%%%%%%%%%%%%%

\section{Introduction}

By utilizing classical and geometric invariant theory techniques, we can establish the following identities:
%\todo[inline]{Pensar en poner $\left(\Proj \Sim^\bullet\Hom(\Bbbk^n, \Bbbk^m)^{\vee}\right) / \SLm$ y citar \cite{invariant-dolgachev}. Aquí tenemos exactamente este ejemplo con el dual (pp. 4) y la acción por la izquierda.}
\begin{equation}
    \label{E:Proj/SL is Proj^SL}
    \left(\Proj \Sim^\bullet\Hom(\Bbbk^m, \Bbbk^n)\right) / \SLm \simeq \Proj (\Sim^\bullet\Hom(\Bbbk^m,\Bbbk^n))^{\SLm}
    \simeq \operatorname{Grass}^m(\Bbbk^n)
\end{equation}

Furthermore, the Plücker morphism provides a $\SL_{n}$-equivariant
%\todo{es cierto que es GL-equiv, pero en \S2 solo se habla de SL-equiv, mejor SL? por coherencia?} 
closed embedding of the aforementioned scheme into the projective space $\P (\wedge^m \Bbbk^n)^\vee$, where the image is given by the intersection of quadrics. The proof of these claims can be found in~\S\ref{sec:toy model}, under the assumption that $\Bbbk$ is a field and $n\geq m$.
%(potentially algebraically closed\todo{por qué esta condición de alg. cerrado?. En \S2 no aparece, propongo quitarlo de aquí y poner alg. closed en \S2}) and $n\geq m$.

These facts serve as a foundational example for the problems addressed in this paper. To establish a precise framework, we consider the following scenario.

\begin{assumption}\label{assumptions}
Let $R$ be an algebra over a ring $\Bbbk$, $T$ an $R$-algebra, $M$ a finitely generated projective $R$-module, and $N$ a $T$-module. Let $G$ be a linearly reductive group scheme over $\Bbbk$ equipped with a representation $\rho:\underline{G}_{R}\rightarrow \underline{\Aut}_{\Mod(R)}(M)$, representing a natural transformation between the group functors. Finally, consider the graded $T$-algebra

\begin{equation}
    \label{E:definition A as invariants}
     A \,:=\, \left( \Sim_{T}^\bullet (M^{\vee} \otimes_{R}N )\right)^{G}       
\end{equation}
where $M^\vee:=\Hom_{\Mod(R)}(M,R)$ denotes the dual of $M$.
\end{assumption}

This paper primarily focuses on two main issues. Firstly, we determine the conditions under which the graded $T$-algebra $A$ is finitely generated, finitely presented, or flat. Secondly, we establish the conditions under which a closed embedding of $\Proj A$ into a projective space exists. Notably, we do not impose any Noetherian hypotheses. Next, we discuss both of these issues separately.

Regarding the properties of $A$ in \eqref{E:definition A as invariants}, let us recall that a central problem in classical invariant theory can be formulated as follows: (1) Given a base field $\Bbbk$, a algebraic group $G$ over$\Bbbk$ and a finite-dimensional $\Bbbk$-representation $G\rightarrow \GL(V)$, describe the invariant elements of the symmetric algebra of the dual, $(\Sim^{\bullet}(V^{\vee}))^{G}$, in terms of generators and relations. This problem can be equivalently formulated in two other ways: (2) Describe $(V^{\otimes r}\otimes (V^{\vee})^{\otimes s})^{G}$ in terms of generators and relations, and (3) decompose the $r$th tensor power $V^{\otimes r}$ into a direct sum of irreducible representations. The first formulation directly relates to the quotient variety $V/G$, while the third one is connected to representation theory.

Historically, the question of whether $(\Sim^{\bullet}(V^{\vee}))^{G}$ is always a finitely generated $\Bbbk$-algebra was posed by Hilbert in this generality and is now known as the 14th Hilbert problem \cite{hilbert}. Subsequently, several authors, including Schur and Weyl \cite{weyl} (see also \cite{deConcini-Procesi} and references therein), focused on the case of classical groups. In fact, when the algebraic group $G$ is classical (e.g., $\GLm$, $\SLm$, $\operatorname{Sp}_{2m}$, $\operatorname{O}_{m}$, or $\operatorname{SO}_{n}$, among others), the description of the generators is commonly referred to as the First Fundamental Theorem (FFT), while the description of the relations is known as the Second Fundamental Theorem (SFT). These fundamental theorems imply that $(\Sim^{\bullet}(V^{\vee}))^{G}$ is a finitely generated $\Bbbk$-algebra in these cases.

However, Nagata presented an example where $A^{G}$ fails to be a finitely generated subring of $A$ (\cite{nagata-hilbert}). He also demonstrated that there is a class of groups, called reductive groups, for which the subring consisting of invariant elements $A^{G}\subset A$ is a finitely generated $\Bbbk$-algebra for any finitely generated $\Bbbk$-algebra $A$ and any representation $\rho:G\rightarrow \Aut_{\Alg(\Bbbk)}(A)$ (\cite{nagata}). This fundamental result forms the basis of geometric invariant theory and constructive methods in many moduli problems \cite{mumford-git}.

Later, Seshadri proved a natural generalization of Nagata's result \cite[Theorem 2]{seshadri-relative}. 
Indeed, Seshadri's results imply some finiteness properties of the subalgebra of invariant elements under the action of $G$ provided that $G$ is a reductive group scheme over $R$ and  $R$ is of finite type over a universally Japanese ring $\Bbbk$.

The first main result of this article,  based on a generalization of Seshadri's result to a relative and non-necessarily Noetherian context (Theorem~\ref{t:Invariant de Sim de Mm es fin pres}), reads as follows.

\begin{theorem}\label{t:main-fin-pres}
Consider $(\Bbbk, R, M, T, N, G, A)$ as described in Assumptions~\ref{assumptions}.

If $N$ is a finitely generated (resp. finitely presented, flat) $T$-module, then $A$ is a finitely generated (resp. finitely presented, flat) $T$-algebra.
\end{theorem}

Now, let us address the second issue mentioned at the beginning of this introduction: establishing conditions that imply the embedding of $\Proj A$ into a projective space.

A fundamental question in algebraic geometry is to determine when a given scheme is (quasi-)projective. As mentioned earlier, the Grassmannian $\Bbbk$-scheme~\eqref{E:Proj/SL is Proj^SL} mentioned at the beginning of this introduction is a homogeneous space with respect to the action of $\GL_{n}$. Therefore, in addition to the Plücker morphism, one can apply certain general results (see, for instance, \cite[Corollaire VI.2.6]{Raynaud}, \cite[Theorem 5.2.2]{Brion}) to obtain an equivariant compactification by a projective scheme. Among other related results, Nagata's compactification theorem is noteworthy. It states that any separated finite-type scheme over a Noetherian base admits an open immersion into a proper scheme over the base \cite{Nagata1, Nagata2}.

However, recall that in our study of $A$ (see~\eqref{E:definition A as invariants}), we do not want to impose any Noetherian or finiteness constraints. In order to handle this general situation, we introduce a new type of graded algebras called partially generated graded algebras (pgg). This type of algebra generalizes the notion of finite presentation since every graded $R$-algebra of finite presentation is a pgg (Proposition~\ref{p:finpres-pgg}) and, conversely, every Noetherian pgg-$R$-algebra is finitely presented (Corollary~\ref{C:noether+pgg imply finite pres}). 
Next, we prove a salient feature of these algebras: the projective spectrum (resp. spectrum) of a pgg-algebra admits a canonical closed immersion (resp. locally closed) into a projective space, which is the projectivization of a module, not necessarily of finite type.

Our second main result is as follows.

\begin{theorem}\label{T:closed immersion Proj affine}
    Let us consider $(\Bbbk, R, M, T, N, G, A)$ as described in Assumptions~\ref{assumptions}. Assume that one of the following conditions holds:
\begin{enumerate}
    \item $N$ is a $T$-module of finite presentation, 
    \item  $\Bbbk$ is a equi-characteristic zero noetherian commutative ring, $R=\Bbbk$, and $G$ is one of the four groups in~\eqref{E:classical groups},
    \item  $G$ is a finite group scheme over $\Bbbk$ (in this case, $\vert G\vert$ is invertible in $\Bbbk$).
\end{enumerate}

Then, there exists a natural number $t\geq 2$ such that 
	\begin{equation}\label{e:projectiveembeddingforpgg-intro}
	    \Phi:\Proj A \,\hookrightarrow\, 
	{\mathbb P} ( [ \Sim^{\bullet}(A_{\leq t} )]_{(t+1)!}^{\vee} )
	\end{equation}
is a canonical closed embedding of $T$-schemes fulfilling the following properties:
\begin{enumerate}
    \item $\Phi$ is functorial on $T$.
    \item $\Phi$ is equivariant with respect to the natural action of  $\Aut_{\gAlg(T)}(A)$. In particular, it is equivariant with respect to the induced action of $\Aut_{\Mod(T)}(N)$.
    \item $\operatorname{Im}\Phi$ is the intersection of hypersurfaces of degree bounded by $t$.
\end{enumerate}
 Moreover, in case $(2)$, $t=2 m$ if $G=\SL(M)$, and $t=2 + m$ for the other groups ($m=\underset{x\in\Spec(\Bbbk)}{\textrm{max}}\{\rk_{\Bbbk(x)}(M_{\Bbbk(x)})\}$), whereas in case $(3)$, $t=2 \vert G \vert$. 
\end{theorem}

Regarding the morphism~\eqref{e:projectiveembeddingforpgg-intro}, two comments are in order. Firstly, it can be used to study compactifications of homogeneous spaces (Remark~\ref{R:PGL/G}).Second, particular instances of it have already been considered when dealing with certain moduli spaces (Remark~\ref{R:applications moduli Schmitt}). 

The article is structured as follows.~\S\ref{sec:toy model} is concerned with  the example presented at the beginning and sketches proofs of the identities \eqref{E:Proj/SL is Proj^SL} and of their associated properties.

In~\S\ref{sec:preliminaries}, we introduce preliminary notions regarding representations of group schemes. It is noteworthy that throughout the paper, we assume the group $G$ to be linearly reductive (see, for instance, \cite{alper}). This assumption is motivated by two reasons: firstly, taking invariants by a linearly reductive group is an exact functor, and secondly, in characteristic zero, linear reductiveness coincides with reductiveness. For precise statements and references, please refer to~\S\ref{subsec:linearly red groups}.

Section~\ref{sec:finitess invariants} is dedicated to the proof of Theorem~\ref{t:main-fin-pres} (see~\S\ref{subsec:first main theorem}). The proof relies on a technical result that is a significant generalization of Seshadri's theorem \cite[Theorem 2]{seshadri-relative} to the relative and non-Noetherian setting. Specifically, employing standard Noetherian reduction arguments, we establish Theorem~\ref{t:Invariant de Sim de Mm es fin pres}, which asserts that $\big( \Sim_{R}^\bullet (M^{\vee})^{\oplus m}\big)^{G}$ is a finitely presented $R$-algebra under the assumptions of~\ref{assumptions}.

In~\S\ref{sec:pgg}, we introduce the concept of a partially generated graded algebra (see~\S\ref{subsec:definition pgg}). We define a graded $R$-algebra as a $t$-pgg-algebra if it is generated by elements of degree less than or equal to $t$, and the ideal of relations between the generators satisfies the same condition. Furthermore, for a given $t$-pgg $R$-algebra $A$, we establish that the $R$-scheme $\Proj A$ (resp. $\Spec A$) admits a canonical equivariant (resp. locally) closed embedding into a projective space (see~\S\ref{subsec:projective embd pgg}).

Section~\ref{sec:algebra of invariants} is devoted to the proof of Theorem~\ref{T:closed immersion Proj affine} and other related results. The key step involves demonstrating that $A$ \eqref{E:definition A as invariants} is a $t$-pgg algebra for a specific value of $t$, followed by the application of the results from~\S\ref{sec:pgg}. This section concludes by addressing potential applications of our findings. Notably, the constructions of moduli spaces of principal bundles based on tensors (e.g., \cite{bhosle,langer,munoz1,schmitt}) can significantly benefit from our results and techniques, as our work establishes new properties (e.g., flatness, functoriality, canonicity) of the main objects presented in those papers.

Throughout the paper, we fix a commutative ring $\Bbbk$. By a $\Bbbk$-algebra, we refer to a commutative $\Bbbk$-algebra. For a commutative ring $R$, $\Alg(R)$ denotes the category of commutative $R$-algebras, and $\Mod(R)$ denotes the category of $R$-modules. In this context, $\Sim_R^{\bullet}(M)$ represents the symmetric algebra of an $R$-module $M$, defined as the quotient of the tensor algebra of $M$ by the ideal generated by $m\otimes m' - m'\otimes m$ for $m, m'\in M$ (see \cite[Chapter II, \S1.7]{ega2}). We denote $\P(M^{\vee})$ as the homogeneous spectrum $\Proj(\Sim_R^{\bullet}(M))$. Lastly, by group scheme over a $\Bbbk$-algebra $R$ we mean a faithfully flat, finitely presented, and separated affine group $R$-scheme.

\section{A Toy Model}\label{sec:toy model}

In order to be self-contained, let us sketch the proof of~\eqref{E:Proj/SL is Proj^SL} and the claims related to it. Let $(\Bbbk, R, M, T, N, G, A)$ be as in Assumptions~\ref{assumptions} and assume furthermore that $\Bbbk$ is a (algebraically closed) field, $R=T= \Bbbk$,  $M=\Bbbk^{m}$, $N=\Bbbk^{n}$ with $n\geq m$ and $G=\SLm$.  
In terms of classical invariant theory, $A$ is isomorphic to the subring of invariants of the ring of polynomial functions of the set of matrices $\operatorname{Mat}_{m, n}(\Bbbk)$ with respect to action on $\SLm$ on the right, that is, 
\[
A=(\Sim^{\bullet}\Hom(\Bbbk^{m},\Bbbk^{n}))^{\SLm}.
\]
Note, however, that the set $\operatorname{Mat}_{m, n}(\Bbbk)$ can be thought as the set of rational points of the scheme of homomorphisms 
\[
\HOM(\Bbbk^{n}, \Bbbk^{m}):= \Spec \Sim^{\bullet}\Hom(\Bbbk^{m},\Bbbk^{n})
\]
so that geometric invariant theory can be applied. 

We recall now the main algebraic and geometric properties of the $\Bbbk$-algebra $A$.

\subsection{Generators and relations} (see~\cite[\S3]{deConcini-Procesi}). For each multiindex $J=(i_1,\hdots,i_m)$ with $1\leq i_j\leq  n$, we define $d_{J}\in \Sim^{\bullet}\Hom(\Bbbk^{m},\Bbbk^{n})$ as the minor of the general matrix formed by the columns\ indexed by the elements of $J$. These functions are not identically zero if there are no repeated indexes. Clearly, the polynomial functions $d_{J}$ are $\SLm$ invariant and, furthermore, $(\Sim^{\bullet}\Hom(\Bbbk^{m},\Bbbk^{n}))^{\SLm}$ is generated by the (finite) set $\bigwedge(m,n):=\{d_{J}| \ J=(i_1,\hdots,i_m) \textrm{ with } 1\leq i_1 < \cdots < i_{m}\leq n\}$.
    Now, for each pair of multiindexes $J=(i_1,\hdots,i_{m-1}), \ K=(k_1,\hdots,k_{m+1})$, it holds 
    \begin{equation}
        \label{E:Plucker relation}
     \sum_{u=1}^{m+1}(-1)^{u}d_{J\cup \{k_u\}}\cdot d_{K\setminus \{k_u\}}=0.   \end{equation}

    If we denote by $I$ the ideal of $\Bbbk\langle\bigwedge(m,n)\rangle$ generated by all the elements of the form $\sum_{u=1}^{m+1}(-1)^{u}d_{J\cup \{k_u\}}\cdot d_{K\setminus \{k_u\}}$ (we let $J$ and $K$ to vary) we get an isomorphism 
    \begin{equation}
        \label{E:Hom^SLm is k<m,n>/I}
    A=(\Sim^{\bullet}\Hom(\Bbbk^{m},\Bbbk^{n}))^{\SLm}\simeq \Bbbk\left\langle\bigwedge(m,n)\right\rangle/I.    
    \end{equation}

    \subsection{Moduli and GIT quotient} The inclusion of graded algebras $(\Sim^{\bullet}\Hom(\Bbbk^{m},\Bbbk^{n}))^{\SLm}\subset\Sim^{\bullet}(\Hom(\Bbbk^{m},\Bbbk^{n}))$ defines a rational map 
    $$\pi:\Proj(\Sim^{\bullet}(\Hom(\Bbbk^{m},\Bbbk^{n}))) \dashrightarrow \Proj((\Sim^{\bullet}\Hom(\Bbbk^{m},\Bbbk^{n}))^{\SLm})$$
    defined on an open subset $U$. 
    The points of $U$ are called semistable points and are characterized as follows: a rational point $p\in \Proj(\Sim^{\bullet}(\Hom(\Bbbk^{m},\Bbbk^{n})))$ belongs to $U$ if and only if there is a $\SLm$-invariant homogeneous polynomial function not vanishing at $p$. 
    %%%%%
The following conditions are equivalent
\begin{enumerate}
    \item $p$ is semistable in the sense of GIT,
    \item $p$ is stable in the sense of GIT; that is, its orbit under $\SL_m$ is closed and its subgroup of isotropy is finite,
    \item $\widehat{p}:\Bbbk^{n} \to \Bbbk^{m}$ is surjective, $\widehat{p}$ being any representative of $p$ in $\HOM(\Bbbk^{n}, \Bbbk^{m})$, %$\Spec(\Sim^{\bullet}(\Hom(\Bbbk^{m},\Bbbk^{n})))$,
    \item there exists $J$ such that $d_{J}(\widehat{p})\neq 0$, $\widehat{p}$ being any representative of $p$ in $\HOM(\Bbbk^{n}, \Bbbk^{m})$. % $\Spec(\Sim^{\bullet}(\Hom(\Bbbk^{m},\Bbbk^{n})))$.
\end{enumerate}
    %%%%%
    The previous discussion have a bunch of consequences. First, 
    \[
    \left\{\begin{array}{l} \textrm{Surjective maps}\\ \Bbbk^{n} \to \Bbbk^{m} \end{array}\right\}=\bigcup_{J}\Spec(\Sim^{\bullet}(\Hom(\Bbbk^{m},\Bbbk^{n})))_{d_{J}}\rightarrow U
    \]
    is a geometric quotient for the action of the multiplicative group over $\Bbbk$, that is, a line bundle over $U$. Here the subindex $d_{J}$ denotes localization.
    Secondly, the map 
    \[ 
    U\rightarrow \big(\Proj(\Sim^{\bullet}\Hom(\Bbbk^{m},\Bbbk^{n}))\big)/{\SLm} \,=\,
    \Proj((\Sim^{\bullet}\Hom(\Bbbk^{m},\Bbbk^{n}))^{\SLm})
    \]
    is a geometric quotient  for the action of the group $\SLm$  on the right. That is, $\Proj((\Sim^{\bullet}\Hom(\Bbbk^{m},\Bbbk^{n}))^{\SLm})$ parametrizes closed orbits. Finally, there is a canonical isomorphism 
    \begin{equation}
        \label{e:Proj^SL=Grass-intro}
    \Proj((\Sim^{\bullet}\Hom(\Bbbk^{m},\Bbbk^{n}))^{\SLm})\simeq \textrm{Grass}_{m}(\Bbbk^{n}),
    \end{equation}
    the later space being the Grassmannian of $m$-dimensional quotients  of $\Bbbk^{n}$.
    
    \subsection{Projective embedding} ~\eqref{E:Hom^SLm is k<m,n>/I} induces a surjection $\Bbbk\langle \bigwedge (m,n)\rangle\rightarrow (\Sim^{\bullet}(\Hom(\Bbbk^{m},\Bbbk^{n})))^{\SL_m}$ and, therefore, a closed embedding (Pl\"ucker embedding)
    \begin{equation}
        \label{e:Plucker for Grass-intro}
    \operatorname{Grass}_{m}(\Bbbk^{n})
    %\simeq \Proj((\Sim^{\bullet}\Hom(\Bbbk^{m},\Bbbk^{n}))^{\SLm})
    \hookrightarrow  \Proj \ \Bbbk\langle \bigwedge (m,n)\rangle\simeq 
    \P ((\wedge^{m} \Bbbk^{n})^{\vee}).
    \end{equation}
    Noting~\eqref{E:Plucker relation}, one concludes that the image is given by intersection of quadrics; namely, the intersection of the zero sets of a generator system of $I$. Furthermore, the Plücker morphism is $\SL_{n}$-equivariant where $\SL_{n}$ acts on the left. 

\subsection{Comments} 
Many aspects of the previous discussion deserve further explanation. It is worth noting that the existence of a system of generators for $A$ that are homogeneous of degree $m$ (see~\eqref{E:Hom^SLm is k<m,n>/I}) and the fact that the ideal of relations among them is generated by elements of degree two (see~\eqref{E:Plucker relation}) are closely linked to our concept of partially generated graded algebra introduced in~\S\ref{sec:pgg}. Additionally,~\eqref{e:Proj^SL=Grass-intro} and~\eqref{e:Plucker for Grass-intro} can be generalized for any ring $\Bbbk$ and any $\Bbbk$-module $N$ (see \cite[Chapter I, \S9.7, \S9.8]{EGA-I-springer}). An analogous situation also exists for the case of the Sato Grassmannian (see Example~\ref{ex:sato}) which is not a scheme of finite type. We will revisit this type of construction for arbitrary schemes in \S\ref{sec:algebra of invariants} (for a connection with compactifications of homogeneous spaces, see Remark~\ref{R:PGL/G}).

    %%%%%%%%%%%%%%%%%%%%%%%%%%%%%%%%%%%%%%%%

\section{Preliminaries on representations}\label{sec:preliminaries}

Let $G$ be a group scheme over $\Bbbk$. For each  $\Bbbk$-algebra $R$, we denote by $G_{R}$ the group $R$-scheme obtained from $G$ by base change respect to the  morphism  $\Spec(R) \to \Spec(\Bbbk)$. The ring of functions of $G_{R}$ will be denoted by $R[G]$ and its associated group-functor by $\underline{G}_{R}$, so $G_{R}=\Spec (R[G])$ and $R[G]=\Bbbk[G]\otimes_{\Bbbk}R.$ We will drop the subscript in $G_{R}$ if $R=\Bbbk$.

%%%%%%%%%%%%%%%%%%%%%%%%%%%%%%%%%%%%%%%%%%%%%%%%%%%%%%%%%%%%%%%%%%%%

\subsection{Representations}
\label{subsec:representations}

In this section we establish the notations, and state the basic notions, about representations of group schemes. We we will follow the fundamental works \cite{jantzen,seshadri-relative}.

Given a $\Bbbk$-algebra $R$ and an $R$-module $M$, we denote by $\underline{\End}_{\Mod(R)}(M)$ the functor of endomorphisms  of $M$:
\begin{equation}\label{e:functorEnd}
    \begin{aligned}
    \underline{\End}_{\Mod(R)}(M):\Alg(R) & \to \text{Set} \\
    T & \leadsto  \End_{\Mod(T)}(M\otimes_{R}T)
    \end{aligned}
\end{equation}
The functor of automorphisms of $M$, $\underline{\Aut}_{\Mod(R)}(M)$, is defined analogously.

If $M$ is a finitely generated and projective $R$-module, the symmetric algebra $\Sim_{R}^{\bullet}(\End_{\Mod(R)}(M)^{\vee})$ is a representative of $\underline{\End}_{\Mod(R)}(M)$, that is, for each $R$-algebra $T$, there is an isomorphism
    \[
    \underline{\End}_{\Mod(R)}(M)(T) \,\simeq\, \Hom_{\Alg(R)}( \Sim_{R}^{\bullet}(\End_{\Mod(R)}(M)^{\vee}) , T)
    \]
which behaves functorially on $T$. Furthermore,  we may localize $\Sim_{R}^{\bullet}(\End_{\Mod(R)}(M)^{\vee})$  by the determinant function, and the functor of points of the corresponding affine $R$-scheme is exactly $\underline{\Aut}_{\Mod(R)}(M)$. 
If $M$ is trivial of rank $m$, we use the standard notation $\underline{\GLm}(R)$ instead of $\underline{\Aut}_{\Mod(R)}(R^{\oplus m})$.

A linear representation of a group $\Bbbk$-scheme, $G$,  over $R\in \Alg(\Bbbk)$, is a pair $(M,\rho)$ where $M$ is an $R$-module, and $\rho:\underline{G}_{R}\rightarrow \underline{\Aut}_{\Mod(R)}(M)$ is a morphism of group functors. Note that giving $\rho$ is the same as giving a morphism of groups $\rho_{T}:\underline{G}_{R}(T)\rightarrow \underline{\Aut}_{\Mod(T)}(M\otimes_{R}T)$ for each $R$-algebra $T$ in a functorial way. In such situation, it is said that $M$ is $G$-$R$ module.
If $M$ has an additional structure of $R$-algebra, then a representation of $G$ over $R$ is a morphism $\rho$ as above whose image lies in $\underline{\Aut}_{\Alg(R)}(M)$. 
In any case, we say that $G_{R}$ acts on $M$.

Furthermore, the datum $\rho$ yields a morphism of $R$-modules:
\begin{equation}\label{E:co-module}
\begin{split}
\tilde\rho:M&\rightarrow M\otimes_{R}R[G]\\
m&\mapsto \rho_{R[G]}(\Id_{G_{R}})(m\otimes 1),
\end{split}
\end{equation}
which makes $M$ a co-module over the co-algebra $R[G]$. 

Finally, let us discuss the case where $M$ is finitely generated and projective. On the one hand, giving a morphism $M\rightarrow M\otimes_{R}R[G]$ is, in turn, equivalent to giving a morphism $R$-modules:
$$
\End_{\Mod(R)}(M)^{\vee}\rightarrow R[G]=\Bbbk[G]\otimes_{\Bbbk}R.
$$
Taking $\Spec$ on the induced morphism of algebras $\Sim_{R}^{\bullet} (\End_{\Mod(R)}(M)^{\vee})\rightarrow R[G]$
yields the representation $\rho:\underline{G}_{R}(T)\rightarrow \underline{\Aut}_{\Mod(R)}(M)$.

%\todo[inline]{Lo pongo para comentar: en lo anterior, supongo qe estás haciendo uso de $\Hom_{\Mod(R)}(M\otimes_{R} M^{\vee} , R[G]) = \Hom_{\Mod(R)}(M,   \Hom_{\Mod(R)}( M^{\vee} , R[G])) = \Hom_{\Mod(R)}(M,   M^{\vee\vee } \otimes_{R} R[G]))$ que se identifica con $\Hom_{\Mod(R)}(M,   M \otimes_{R} R[G]))$ cuando $M$ es finito y proyectivo.} 
On the other hand, since   there is a canonical isomorphism $M^{\vee}\otimes_{R}T\simeq (M\otimes_{R}T)^{\vee}$, 
any representation $\rho:\underline{G}_{R}\rightarrow\underline{\Aut}_{\Mod(R)}(M)$ induces naturally a representation (the contragradient representation) $\rho^{\vee}:\underline{G}_{R}\rightarrow 
\underline{\Aut}_{\Mod(R)}(M^{\vee})$ defined by the rule:
\begin{equation}\label{E:contragradient}
\begin{split}
\rho_{T}^{\vee}(g): (M\otimes_{R} T)^{\vee} & \rightarrow (M\otimes_{R}T)^{\vee}\\
\varphi &\mapsto  \varphi \circ \rho_{T}(g), 
\end{split}
\end{equation}
and the co-module structure induced on $M^{\vee}$ by the contragradient representation is given by:
\begin{equation}\label{co-grad}
\begin{split}
M^{\vee} & \rightarrow  \Hom_{\Mod(R)}(M,R[G]) \simeq M^{\vee}\otimes_{R}R[G] \\
\delta &\mapsto  (\delta\otimes 1) \circ \rho .
\end{split}
\end{equation}

%%%%%%%%%%%%%%%%%%%%%%%%%%%%%%%%%%%%%%%%%%%%%%%%%%%%%%%%%%%%%%%%%%%%

\subsection{Linearly reductive groups}
\label{subsec:linearly red groups}

Given a $\Bbbk$-algebra $R$ and a group scheme $G$ over $R$, we denote by $BG$ the Artin stack $[\Spec(R)/G_{R}]$. It carries a structure morphism $f:BG\rightarrow \Spec(R)$.
\begin{definition}\cite[Definition 12.1]{jalper}
The group scheme $G$ is called linearly reductive if $f:BG\rightarrow \Spec(R)$ is cohomologically affine. 
\end{definition}

If we are given an $R$-module $M$ and a representation $\rho:\underline{G}\rightarrow\underline{\Aut}_{\Mod(R)}(M)$ (i.e., $M$ is a $G$-$R$ module), we denote by $M^{G}$ the $R$-submodule of $M$ consisting of invariant elements. The operation ``taking invariants by $G$'' defines a functor from the category of $G$-$R$-modules to the category of $R$-modules
\begin{equation}
\begin{split}
(-)^{G}:\Mod(G\operatorname{-}R)&\rightarrow \Mod(R)\\
M&\leadsto M^{G} ,
\end{split}
\end{equation}
which is left exact since $G$ is assumed to be flat over $R$.

\begin{remark}
Let $T$ be an $R$-algebra and $N$ a $G_{T}$-$T$-module. For the sake of notation, we write $N^{G}$ for the submodule of $G_T$-invariants of the $T$-module $N$.
\end{remark}

Linearly reductive group schemes have been studied exhaustively and characterized in many ways (see for instance \cite{abramovich,alper,sancho,kemper-linear,Springer}).
We summarize now some of the properties of linearly reductive group schemes that will be used along this work.
\begin{proposition}\label{prop:linearly}
Let $G$ be a group scheme over $R$. Then, the following properties hold true:
\begin{enumerate}
\item $G$ is linearly reductive if and only if $(-)^{G}$ is exact.
\item If $G$ is linearly reductive and $T$ is an $R$-algebra, then the base change $G_{T}:=\Spec(T[G])$ is linearly reductive.
\item If $G$ is linearly reductive, $M$ a $G$-$R$ module and $T$ an $R$-algebra, then $M^{G}\otimes_{R}T\simeq (M\otimes_{R}T)^{G}$.
\item If $G$ is linearly reductive, $R$ is noetherian and $T$ is a finitely generated $R$-algebra acted on by $G_T$, then $T^{G}$ is a finitely generated $R$-algebra.
\item If the fibers of $G$ over closed points are linearly reductive and $R$ is noetherian, then $G$ is linearly reductive. 
%\todo{falta decir $G$ lin. reduc., no? COMO TENEMOS (2), CREO QUE NO HACE FALTA}
\end{enumerate}
\end{proposition}
\begin{proof}
Properties (1), (2) and (3) can be found in \cite{jalper}. Linearly reductive group schemes over a noetherian base are power-reductive \cite{Franjou}, so (4) follows from \cite[Proposition 7]{power-red}. Statement (5) follows from \cite[Theorem 1.2]{Margaux} and (1).
\end{proof}

Given a $G$-$R$-module $M$, a Reynolds operator is a $G$-equivariant and $R$-linear surjection, 
$$\mathcal{R}:M\rightarrow M^{G},$$ 
giving the identity when is restricted to $M^{G}$.
In particular, this makes $M^{G}$ a direct summand of $M$.
If $G$ is linearly reductive, there is always a Reynolds operator for a given $R$-module (see \cite[Theorem 2]{sancho}, \cite[Theorem 1.10]{sancho2}). 

As examples of linearly reductive groups we mention finite groups $G$ over $\Bbbk$ such that $\vert G \vert$ is invertible in $\Bbbk$. Recall that, over a field of characteristic $0$, a group is linearly reductive if and only if it is reductive. Thus, the groups $\GLm$, $\SLm$, $\operatorname{Sp_{2m}}$, $\operatorname{O}_{m}$ and $\operatorname{SO}_{m}$ (see~\ref{E:classical groups}) are examples of linearly reductive groups over a equi-characteristic noetherian commutative ring $\Bbbk$ of characteristic $0$.

\begin{remark}
An important advance in relation to 14th Hilbert's problem is found in the works of Franjou and van der Kallen, where the notion of power-reductive group is introduced. In \cite{Franjou}, the authors prove that $A^{G}$ is finitely generated for any $A$ (finitely generated over Noetherian $\Bbbk$) if and only if $G$ is power-reductive. It is important to note that in this article we are dealing with linearly reductive rather than power-reductive groups $G$ because the Noetherian induction arguments require that the operation "take invariants" have nice base change properties. This is true for linearly reductive groups but not, in general, for other types of groups (see \cite{kemper-linear}). 

It is important to note that, although it is true that in positive characteristic the condition of being linearly reductive is very restrictive (they are finite groups, tori and finite products of these, \cite[Théoréme de Nagata]{gabriel}), in characteristic zero the properties linearly reductive, geometrically reductive, reductive and power-reductive are all the same.
\end{remark}

%%%%%%%%%%%%%%%%%%%%%%%%%%%%%%%%%%%%%%%%%%%%%%%%%%%%%%%%%%%%%%%%%%%%
%%%%%%%%%%%%%%%%%%%%%%%%%%%%%%%%%%%%%%%%%%%%%%%%%%%%%%%%%%%%%%%%%%%%

\section{Finiteness properties of the subalgebra of invariants}\label{sec:finitess invariants}

This section proves Theorem~\ref{t:main-fin-pres}. To achieve this, we begin generalizing to an arbitrary base (i.e. non necessarily noetherian) Proposition~\ref{prop:linearly} (4). Because of the interest in its own, due to the historical link with Hilbert 14th problem, we devote a first subsection to this generalization (see Theorem~\ref{t:Invariant de Sim de Mm es fin pres}).

%%%%%%%%%%%%%%%%%%%%%%%%%%%%%%%%%%%%%%%%%%%%%%%%%%%%%%%%%%%%%%%%%%%%
%%%%%%%%%%%%%%%%%%%%%%%%%%%%%%%%%%%%%%%%%%%%%%%%%%%%%%%%%%%%%%%%%%%%

\subsection{Noetherian reduction}

%\todo[inline]{Es la condición más general que podemos poner para hacer uso en Th. 3.8. del resultado de Seshadri. Sería mejor hacer uso de los resultados de "Reductivity properties over an affine base"? Aquí no es necesario dar propiedades geométricas de G, pero en Seshadri hay que decir que $G\rightarrow S$ es liso y de fibras conexas además de reductivas. De cara a las aplicaciones con grupos finitos quizás sería conveniente no tener que imponer demasiadas condiciones geométricas sobre G.}

We recall the fundamental lemmas regarding direct limits, and prove others, that will be needed along this section. If the transition maps $\phi_{ij}$ defining a direct limit $(M_i, \phi_{ij})$ are clearly established, and no confusion arises, we will drop them from the notation for the sake of clarity.

Let $R_0$ be a commutative ring, $\{R_{\lambda}\}_{\lambda\in I}$ a direct system of $R_0$-algebras and $R$ the direct limit $\varinjlim  R_{\lambda}$. If $M_{0}$ is an $R_0$-module, we can set $M_{\lambda}:=M_{0}\otimes_{R_{0}}R_{\lambda}$. Since  $M_{\beta}$ and $ M_{\lambda}\otimes_{R_{\lambda}}R_{\beta}$ are canonically isomorphic, there is a canonical morphism $M_{\lambda}\rightarrow M_{\beta}$ for each pair $\lambda\leq \beta \in I$, namely, $m\mapsto m\otimes 1$. Thus, $\{M_{\lambda}\}_{\lambda\in I}$ forms a direct system and we define $M := \varinjlim  M_{\lambda}$. If $N_{0}$ is another $R_{0}$-module, the same discussion yields a direct system $\{N_{\lambda}\}_{\lambda\in I}$, and we define $N := \varinjlim  N_{\lambda}$. Note that the tensor product operation $-\otimes_{R_{\lambda}}R_{\beta}$ induces a morphism
$$
\Hom_{\Mod(R_{\lambda})}(M_{\lambda},N_{\lambda})\longrightarrow \Hom_{\Mod(R_{\beta})}( M_{\beta},N_{\beta})
$$
for each pair $\lambda\leq \beta \in I$. Thus, we also have a direct system $\{\Hom_{\Mod(R_{\lambda})}(M_{\lambda},N_{\lambda})\}_{\lambda\in I}$ and a canonical morphism
\begin{equation}\label{eq:mor-fp}
\theta : \varinjlim  \Hom_{\Mod(R_{\lambda})}(M_{\lambda},N_{\lambda})\rightarrow \Hom_{\Mod(R)}(N,M)
\end{equation}

\begin{lemma}\cite[Lemma 8.5.2.1]{ega4-3}\label{lm:noeth-red-1}
Let $R_0$ be a commutative ring, $\{R_{\lambda}\}_{\lambda\in I}$ a direct system of $R_0$-algebras, $R=\varinjlim  R_{\lambda}$; let $M_0 , N_0$ be $R_0$-modules, and let us denote $M_{\lambda}=M_{0}\otimes_{R_{0}}R_{\lambda}, N_{\lambda}=N_{0}\otimes_{R_{0}}R_{\lambda}$, $M=M_{0}\otimes_{R_{0}}R=\varinjlim  M_{\lambda}, N=N_{0}\otimes_{R_{0}}R=\varinjlim  N_{\lambda}$. If $M_{0}$ is of finite type (finite presentation), then the canonical homomorphism~\eqref{eq:mor-fp}
is injective (respectively, bijective).
\end{lemma}

%\todo[inline]{¿añadir Remark?, no veo ahora que si un límite por \eqref{eq:mor-fp} es cero, sea cero en una etapa que creo que es lo que se invoca más adelante. Supongo que si $R$ noetheriano y $N_0$ tipo finito, entonces si es cierto.}

\begin{lemma}\cite[Proposition 8.9.1 (ii)]{ega4-3}\label{lm:noeth-red-2}
Let $R$ be a commutative ring and $M$ an $R$-module of finite presentation. Then, there exits a noetherian subring $R_{0}\subset R$ and an $R_0$-module of finite presentation such that $M_{0}\otimes_{R_{0}}R\simeq M$.
\end{lemma}

\begin{remark}\label{rmk:M0 projective}
If $M$ is projective, we may assume that $M_0$ is also projective.
\end{remark}

\begin{lemma}\label{lm:indexes}
Let $R$ be a commutative ring and $R_{0}\subset R$ a noetherian subring. Consider the set 
$J=\{R'\subset R \text{ s.t. } R' \text{ is finitely generated over } R_{0}\}$ together with the preorder defined by inclusion. 
Then, $\varinjlim_{J} R_{\lambda}=R$.
\end{lemma}

%{\color{red}
%\begin{lemma}\label{lm:indexes}
%Let $R$ be a commutative ring. Let $\{R_{\lambda}\}_{\lambda\in I}$ be the direct system formed by the noetherian subrings of $R$, so $R=\varinjlim_{I} R_{\lambda}$. Let $R_{0}\in \{R_{\lambda}\}_{\lambda\in I}$. Consider the set of indexes $J=\{\lambda\in I | R_{0}\subset R_{\lambda}\}$. Then, $\varinjlim_{J} R_{\lambda}=R$.
%\end{lemma}
%}

\begin{proof}
%Consider the set $I=\{R'\subset R | \ R' \textrm{ is finitely generated over $\mathbb{Z}$}\}$ together with the preorder defined by inclusion. It is well-known that $\varinjlim_{I} R_{\lambda}=R$. 

Since $\varinjlim_{J}  R_{\lambda}=\bigcup_{J} R_{\lambda}$, to prove that $J$ is a direct system we just need to show that given $R_{\alpha}$ and $R_{\beta}$ algebras finitely generated over $R_0$, the subring $R_{\gamma}$ generated by $R_{\alpha}$ and $R_{\beta}$ is finitely generated over $R_0$. Let us fix presentations $u_{\alpha}:R_0[x_1,\hdots,x_m]\rightarrow R_{\alpha}\rightarrow 0$ and $u_{\beta}:R_0[y_1,\hdots,y_n]\rightarrow R_{\beta}\rightarrow 0$. Then, $R_{\gamma}$ is given by the image of the morphism of $k$-algebras
$$
u:R_0[x_1,\hdots,x_m,y_1,\hdots,y_n]\rightarrow R
$$
defined by $u(x_{i})=u_{\alpha}(x_{i})$ and $u(y_{j})=u_{\beta}(y_{j})$. Thus, $R_{\gamma}$ is finitely generated over $R_0$ and, clearly, $R_{\alpha}, R_{\beta}\subset R_{\gamma}$. 

Finally, note that for any $r\in R$, $R_0[r]$ is a finitely generated $R_0$-algebra and, thus, $\varinjlim_{J} R_{\lambda}=R$.
\end{proof}

\begin{lemma}\label{lm:noeth-red-3}
Let $R$ be a commutative ring and $M$ an $R$-module of finite presentation. 
 Let $R_{0}$ be a noetherian subring of $R$ and let $M_{0}$ be an $R_0$-module of finite presentation such that $M_{0}\otimes_{R_{0}}R\simeq M$.
Then, there exits
a direct system of noetherian $R_{0}$-subalgebras of $R$, $\{R_{\lambda}\}_{\lambda\in J}$, and a direct system of modules of finite presentation, $\{M_{\lambda}\}_{\lambda\in J}$ ($M_{\lambda}$ being an $R_{\lambda}$-module), such that $R=\varinjlim  R_{\lambda}$, $M=\varinjlim  M_{\lambda}$ and $\End_{\Mod(R)}(M)=\varinjlim  \End_{\Mod(R_{\lambda})}(M_{\lambda})$.
\end{lemma}
\begin{proof}
Let $J$ and $\{R_{\lambda}\}_{\lambda\in J}$ be as in  Lemma~\ref{lm:indexes} and denote $M_{\lambda}=M_{0}\otimes_{R_{0}}R_{\lambda}$. Note that $M=\varinjlim M_{\lambda}$, since direct limits commutes with tensor product. Now, the result follows directly from Lemma~\ref{lm:noeth-red-1}. Note that, for each $\lambda$,  the $R_0$-algebra $R_{\lambda}$ is a noetherian subring of $R$ since it is finitely generated over $R_{0}$. 
\end{proof}

\begin{remark}
The above lemmas can be stated in terms of $\Bbbk$-algebras instead of just commutative rings.  We will use this fact without any further mention to this remark.
\end{remark}
\begin{proposition}\label{prop:noeth-red}
Let $R$ and $A$ be $\Bbbk$-algebras, and $M$ a finitely generated projective $R$-module. Let $f:\End_{\Mod(R)}(M)^{\vee}\rightarrow A\otimes_{\Bbbk}R$ be a morphism of $R$-modules. Then, there exits a noetherian $\Bbbk$-subalgebra $R_{0}\subset R$, a finitely generated projective $R_{0}$-module $M_0$, and a morphism of $R_{0}$-modules $f_{0}:\End_{\Mod(R_0)}(M_0)^{\vee}\rightarrow A\otimes_{\Bbbk}R_{0}$ such that $M_0 \otimes_{R_0} R=M$ and $f_{0}\otimes 1=f$.
\end{proposition}
\begin{proof}
 Let $R_0$ and $M_0$ be as in Lemma~\ref{lm:noeth-red-2}. Let $\{R_{\lambda}\}_{\lambda\in J}$ and $\{M_{\lambda}\}_{\lambda\in J}$ be as in Lemma~\ref{lm:noeth-red-3}. By Remark~\ref{rmk:M0 projective}, we may assume that $M_\lambda$ is projective for every $\lambda$. Note that, under these conditions there is a canonical isomorphism $\End_{\Mod(R_{\alpha})}(M_{\alpha})\otimes_{R_{\alpha}}R_{\beta}=\End_{\Mod(R_{\beta})}(M_{\beta})$ for every $\alpha\leq \beta\in J$. So, $\varinjlim \End_{\Mod(R_{\lambda})}(M_{\lambda})^{\vee}=\End_{\Mod(R)}(M)^{\vee}$.
On the other hand, since direct limits commute with tensor product, we may write $A\otimes_{\Bbbk}R=\varinjlim (A\otimes_{\Bbbk}R_{\lambda})$. Then, $f$ belongs to
$$
\Hom_{\Mod(R)}(\varinjlim  \End_{\Mod(R_{\lambda})}(M_{\lambda})^{\vee},\varinjlim  (A \otimes_{\Bbbk}R_{\lambda})).
$$
By Lemma~\ref{lm:noeth-red-1}, the above is isomorphic to
$$
\varinjlim  \Hom_{\Mod(R_{\lambda})}(\End_{\Mod(R_{\lambda})}(M_{\lambda})^{\vee},A \otimes_{\Bbbk}R_{\lambda}),
$$
and, thus, there exists $\lambda\in J$ such that $f$ coincides with the class of a map $f_{\lambda} :\End_{\Mod(R_{\lambda})}(M_{\lambda})^{\vee}\rightarrow A\otimes_{\Bbbk}R_{\lambda}$. Then, the desired data is given by the triple $(R_{\lambda}, M_{\lambda}, f_{\lambda})$.
\end{proof}

\begin{theorem}\label{t:Invariant de Sim de Mm es fin pres}
Let $(\Bbbk, R, M, T, N, G, A)$ be as in Assumptions~\ref{assumptions} and assume $T=R$, $N=T^{\oplus n}$ for a positive integer number $n$. 

Then $A:=  \big( \Sim_{R}^\bullet ((M^{\vee})^{\oplus n})\big)^{G}$ is an $R$-algebra of finite presentation for all $n>0$. 
\end{theorem}

\begin{proof}
The proof consists of a combination of noetherian reduction arguments (\cite{ega4-3}) and~\cite[Theorem 8]{Franjou}.

We claim that  $R$ can be assumed to be noetherian. 
Let $R[G]$ denote the ring of functions of $G_R$. The action of $G_R$ on $M$ can be understood as a morphism of $R$-modules (see Eq.~\ref{E:co-module})
%\todo{ by the discussion of \S2}
	\[
	\tilde \rho: M \longrightarrow R[G]\otimes_R M.
	\]
Let $\rho:\End_R (M)^{\vee}\to R[G]=\Bbbk[G]\otimes_{\Bbbk}R$ be the corresponding morphism of $R$-modules.
Let $R_0$, $M_0$ and $\rho_0$ be as in Proposition~\ref{prop:noeth-red} and $\{R_{\lambda}\}_{\lambda\in J}$, $\{M_{\lambda}\}_{\lambda \in J}$ as in Lemma~\ref{lm:noeth-red-3}. 
Note that $M_\lambda:=M_0\otimes_{R_0} R_\lambda$ and set $\rho_\lambda:=\rho_0 \otimes 1$ (base change to $R_\lambda$). Note that $R_\lambda$ is noetherian and that $M_\lambda$ is a finitely generated projective $R_\lambda$-module (Remark~\ref{rmk:M0 projective}).

We claim that there exists $\beta$ such that $\tilde \rho_\lambda: M_\lambda \longrightarrow R_\lambda[G]\otimes_{R_\lambda} M_\lambda$ defines an action of $G_{R_\lambda}$ on $M_\lambda$ for each $\lambda\geq \beta$. Namely, we have to check that the following two diagrams,
	\begin{equation}\label{e:action1}
	\xymatrix{
	M_\lambda \ar[r]^-{\tilde\rho_\lambda}  \ar[d]_{\tilde\rho_\lambda} & 
	R_\lambda[G]\otimes_{R_\lambda} M_\lambda  \ar[d]^{1\otimes \tilde\rho_\lambda}
	\\
	R_\lambda[G]\otimes_{R_\lambda} M_\lambda  \ar[r]^-{m \otimes 1} & 
	R_\lambda[G]\otimes_{R_\lambda} R_\lambda[G]\otimes_{R_\lambda}  M_\lambda}
	\end{equation}
and
	\begin{equation}\label{e:action2}
	\xymatrix@C=22pt{
	M_\lambda \ar[r]^-{\tilde\rho_\lambda}  \ar[dr]_{\Id} & 
	R_\lambda[G]\otimes_{R_\lambda} M_\lambda  \ar[d]^{e\otimes \tilde\rho_\lambda}
	\\
	& 
	M_\lambda}	
	\end{equation}
are commutative, where $m$ (resp. $e$) is the multiplication map (resp. identity element) of $G$ for every $\lambda \in J$ large enough. Let us define $N_{\lambda}:=R_\lambda[G]\otimes_{R_\lambda} R_\lambda[G]\otimes_{R_\lambda}  M_\lambda$ and consider the morphisms 
\begin{equation*}
\begin{split}
& \Psi_{\lambda}:=1\otimes \tilde\rho_\lambda\circ \tilde\rho_\lambda - m\otimes 1\circ \tilde\rho_\lambda \in \Hom_{\Mod\{R_{\lambda}\}}(M_{\lambda},N_{\lambda}) , \\
&\Phi_{\lambda}:=e\otimes \tilde\rho_\lambda\circ \tilde\rho_\lambda- \Id  \in \Hom_{\Mod\{R_{\lambda}\}}(M_{\lambda},M_{\lambda}) .
\end{split}
\end{equation*}
Varying $\lambda$, we get direct systems of morphisms $(\Psi_{\lambda})_{\lambda}$ and $(\Phi_{\lambda})_{\lambda}$. By construction (see also~\eqref{eq:mor-fp}), we have
\begin{equation*}
\begin{split}
\theta(\varinjlim(\Psi_{\lambda}))=&1\otimes \tilde\rho\circ \tilde\rho - m\otimes 1\circ \tilde\rho=0 ,
\\
\theta(\varinjlim(\Phi_{\lambda}))=& e\otimes \tilde\rho\circ \tilde\rho- \Id =0 .
\end{split}
\end{equation*}
Thus, by Lemma~\ref{lm:noeth-red-1}, we deduce that $\varinjlim(\Psi_{\lambda})=0$ (respectively,  $\varinjlim(\Phi_{\lambda})=0$), so there exists $\beta$ such that $\Psi_{\lambda}=0$ (respectively, $\Phi_{\lambda}=0$) for each $\lambda\geq \beta$. This proves the claim.

By~\ref{co-grad}, the base change of the co-module structure induced on $M_{\beta}^{\vee}$ by the representation $\rho_
\beta$ is equal to the co-module structure induced on $M^{\vee}_{\lambda}$ by $\rho_{\beta}\otimes 1=\rho_{\lambda}$.
We may define now two (a priori) different direct systems of algebras, 
$\{A_{\lambda}\}_{\lambda\geq \beta}$ and $\{B_{\lambda}\}_{\lambda\geq \beta}$, with
\begin{equation*}
\begin{split}
A_{\lambda}:=&\big(\Sim^\bullet_{R_\lambda} (M_{\beta}^{\vee}\otimes_{R_{\beta}}R_{\lambda})^{\oplus n}\big)^{G_{R_\lambda}}, \textrm{ the transition maps }A_{\lambda}\rightarrow A_{\mu}\textrm{ being} \\ &\textrm{ those induced from the transition maps }R_{\lambda}\rightarrow R_{\mu},
\\
B_{\lambda}:=&\big(\Sim^\bullet_{R_\beta} (M_{\beta}^{\vee})^{\oplus n}\big)^{G_{R_\beta}}\otimes_{R_{\beta}}R_{\lambda}, \textrm{ the transition maps }B_{\lambda}\rightarrow B_{\mu} \\ &\textrm{ being those induced from the transition maps }R_{\lambda}\rightarrow R_{\mu}.
\end{split}
\end{equation*}
Since $M_{\beta}$ is finitely generated and projective, and direct limits commute with tensor products and the functor $(-)^{G}$ (\cite[Lemma 4.17]{jantzen}), it follows that the direct limits of the above directed systems are $\big(\Sim^\bullet_R (M_{\beta}^{\vee}\otimes_{R_{\beta}}R)^{\oplus n}\big)^{G}$ and $\big(\Sim^\bullet_{R_\beta} (M_{\beta}^{\vee})^{\oplus n}\big)^{G_{R_\beta}} \otimes_{R_\beta} R$, respectively. 
Note, however, that $A_{\lambda}=B_{\lambda}$ as subalgebras of $\big(\Sim^\bullet_R (M_{\lambda}^{\vee})^{\oplus n}\big)$ by Proposition~\ref{prop:linearly} (3). It follows that 
\begin{equation}\label{e:R_2R}
\big(\Sim^\bullet_R (M^{\vee})^{\oplus n}\big)^{G}\,\simeq\,
\big(\Sim^\bullet_R (M_{\beta}^{\vee}\otimes_{R_{\beta}}R)^{\oplus n}\big)^{G}
\,=\, \big(\Sim^\bullet_{R_\beta} (M_{\beta}^{\vee})^{\oplus n}\big)^{G_{R_\beta}} \otimes_{R_\beta} R.
\end{equation}
by uniqueness of direct limits.

%{\color{red}
%By~\ref{co-grad}, the base change of the co-module structure induced on $M_{\beta}^{\vee}$ by the representation $\rho_
%\beta$ is equal to the co-module structure induced on $M^{\vee}$ by $\rho_{\beta}\otimes 1=\rho$.
%Now, Proposition~\ref{prop:cambio-base} implies that
%\begin{equation}\label{e:R_2R}
%\big(\Sim^\bullet_R (M^{\vee})^{\oplus n}\big)^{G} \,=\, \big(\Sim^\bullet_{R_\beta} (M_{\beta}^{\vee})^{\oplus n}\big)^{G_{R_\beta}} \otimes_{R_\beta} R.
%\end{equation}
%}
Assume that $\big(\Sim^\bullet_{R_\beta} (M_{\beta}^{\vee})^{\oplus n}\big)^{G_{R_\beta}}$ were a finitely generated
 $R_\beta$-algebra. Then it would be of finite presentation since $R_\beta$ is noetherian. 
Hence, by equation~\eqref{e:R_2R}, $A$ would be a finitely presented $R$-algebra. 
Thus, to conclude it suffices to show that, if $R$ is a noetherian ring, then  $A$ is a finite type $R$-algebra. This is precisely Proposition~\ref{prop:linearly} (4).
%\todo[inline]{Concluir por el Th. 3 de Power Reductivity over an Arbitrary Base, Franjou y van der Kallen, en vez de por el Th. de Seshadri. Así solo tenemos que dedicr en el enunciado que el grupo es linealmente reductivo sin tener que imponer que tiene fibras conexas y lisas. Evitar esta condición puede ser interesante en el caso de que $G$ sea finito.}
\end{proof}

%ALTERNATIVA AL COROLARIO ANTERIOR
As the following results shows, Theorem~\ref{t:Invariant de Sim de Mm es fin pres} is a generalization of \cite[Theorem 2]{seshadri-relative}.

\begin{corollary}
Let $G$ be a linearly reductive group scheme over a ring $\Bbbk$ acting linearly on the $m$-dimensional affine space ${\mathbb{A}}_{\Bbbk}^n$. Let $X=\Spec B$ a closed subscheme of ${\mathbb{A}}_{\Bbbk}^n$ fixed by the action of $G$ and $P$ be a $G-B$-module of finite type over $B$. Then we have
\begin{enumerate}
    \item $B^G$ is an $\Bbbk$-algebra of finite type,
    \item $P^G$ is a $B^G$-module of finite type.
\end{enumerate}
\end{corollary}

\begin{proof}
The hypothesis imply that there is an $G$-equivariant surjection of $\Bbbk$-algebras $\Sim^{\bullet}_{\Bbbk} (\Bbbk^{n}) \to B$. Linearly reductivity (Proposition~\ref{prop:linearly}(1)) imply that $(\Sim^{\bullet}_{\Bbbk} (\Bbbk^{n}))^{G} \to B^{G}$ is also surjective. Letting $T=R=\Bbbk$, $M=\Bbbk$, $N=\Bbbk^n$ and applying Theorem~\ref{t:Invariant de Sim de Mm es fin pres}, one has that $(\Sim^{\bullet}_{\Bbbk} (\Bbbk^{n}))^{G}$ is a finitely presented $\Bbbk$-algebra and the first claim follows. 

The second item is a straightforward consequence of $G$ being linearly reductive. Indeed, noting that there exists a surjection $B^{\oplus p}\to P$ of $B$-modules for certain integer number $p$, the statement follows from Proposition~\ref{prop:linearly}(1).
\end{proof}

% Furthermore, if $B$ is an $R$-algebra, $\rho': G_{R}\rightarrow \Aut_{\Alg(R)}(B)$ another representation and $\Sim^{\bullet}_{R}(M^{\vee})\rightarrow B$ a surjective morphism of $R$-algebras compatible with the two actions, $\rho$ and $\rho'$, then the same result as before holds for $B^{G}$.\todo{¿hace falta, tenemos un análogo? Creo que lo quitaría para aligerar.}

%%%%%%%%%%%%%%%%%%%%%%%%%%%%%%%%%%%%%%%%%%%%%%%%%%%%%%%%%%%%%%%%%%%%
%%%%%%%%%%%%%%%%%%%%%%%%%%%%%%%%%%%%%%%%%%%%%%%%%%%%%%%%%%%%%%%%%%%%
%%%%%%%%%%%%%%%%%%%%%%%%%%%%%%%%%%%%%%%%%%%%%%%%%%%%%%%%%%%%%%%%%%%%

\subsection{Proof of Theorem~\ref{t:main-fin-pres}}\label{subsec:first main theorem}

The proof of Theorem~\ref{t:main-fin-pres} relies on  Theorem~\ref{t:Invariant de Sim de Mm es fin pres}.
%and Proposition~\ref{p:Classicalgroups-invariantcondition}. 
Each claim of the Theorem~\ref{t:main-fin-pres} will be proved in a separate subsection.

\subsubsection{Case $N$ is finitely generated}\label{sss:fin-gen}
Consider a surjective morphism 
    \[
    T^{\oplus n} \to  N \to 0
    \]
for an integer $n>0$. Note that, being $G$ linearly reductive,  the exactness of the sequence of graded $T$-modules
	\[
	\Sim_{T}^\bullet (M^{\vee} \otimes_R T^{\oplus n} )
	\to 
	 \Sim_{T}^\bullet (M^{\vee} \otimes_R N)
	 \to 0
	\] 	
yields the following exact sequence
	\[
	B :=  \big( \Sim_{T}^\bullet (M^{\vee} \otimes_R T^{\oplus n})\big)^{G}
 \to A:=  \big( \Sim_{T}^\bullet (M^{\vee} \otimes_R N)\big)^{G} \to 0.
	\]
Thus, it suffices to show that $B$ is finitely generated. One has by Proposition~\ref{prop:linearly}
	\begin{equation}
    \label{E:A^G para N libre}
	B \,=\,   \big( \Sim_{T}^\bullet (M^{\vee} \otimes_R T^{\oplus n})\big)^{G}
	\,\simeq \,  \big( \Sim_{R}^\bullet ((M^{\vee})^{\oplus n})\big)^{G}  \otimes_R T
	\end{equation}
and the conclusion follows  from Theorem~\ref{t:Invariant de Sim de Mm es fin pres}.

\subsubsection{Case $N$ is finitely presented}\label{sss:finitepresented}
Consider a finite presentation of $N$ as a $T$-module
    \[
    N''\to N'\to N\to 0
    \] 
where $N''$ and $N'$ are free finite $T$-modules. Note that, being $G$ linearly reductive,  the exactness of the sequence of graded $T$-modules
	\[
	\widetilde N := (M^{\vee}\otimes_R N'' ) \otimes_{T}  \Sim_{T}^\bullet (M^{\vee} \otimes_R N')
	\to 
	 \Sim_{T}^\bullet (M^{\vee} \otimes_R N')
	\to 
	 \Sim_{T}^\bullet (M^{\vee} \otimes_R N)
	 \to 0
	\] 	
yields the following exact sequence
	\begin{equation}\label{E:casefinpresented1}
	\widetilde N^G
	\to 
	B 	\to A \to 0
	\end{equation}
where $B:=  \big( \Sim_{T}^\bullet (M^{\vee} \otimes_R N')\big)^{G}$ and $A:=  \big( \Sim_{T}^\bullet (M^{\vee} \otimes_R N)\big)^{G}$.

By Theorem~\ref{t:Invariant de Sim de Mm es fin pres} $B$ is a finitely presented $T$-algebra so, to conclude that $A$ is a finitely presented $B$-algebra, it suffices to show that $\widetilde N^G$ is a finitely generated $B$-module.
Indeed, since $\widetilde N$ is a finitely generated module over $\Sim_{T}^\bullet (M^{\vee} \otimes_R N')$, there exist $k\gg0$ and a surjective morphism of $T$-modules
	\[
	\big( \Sim_{T}^\bullet (M^{\vee} \otimes_R N') \big)^{\oplus k} \,\to\, \widetilde N \to 0 .
	\]
which, again by linear reductivity of $G$ and the functoriality of the Reynolds operator, gives rise to a surjective morphism 
\begin{equation}
\label{E:casefinpresented2}
B^{\oplus k} \to \widetilde N^G    
\end{equation}
of $B$-modules and, thus, $\widetilde N^G$ is a finitely generated $B$-module.
%. Noting again that $B$ is a finitely generated $T$-module, the conclusion follows. 

\subsubsection{Case $N$ is flat}\label{subsec:flat}
Recall that a module is flat if and only if it is the direct limit of a direct system of free finite modules (Lazard's Theorem \cite{Lazard}) and that direct limits commute with tensor products. Since $G_T$ acts trivially on $N$, it suffices to show the case where $N$ is a free finite $T$-module. 

 Indeed, for  $N$ a free finite  $T$-module, $T$-flatness of $A$ is equivalent to $R$-flatness of $\big( \Sim_{R}^\bullet ((M^{\vee})^{\oplus n})\big)^{G}$ by \eqref{E:A^G para N libre}. Replacing $M$ if necessary, it suffices to show that $\big( \Sim_{R}^{k} (M^{\vee} )\big)^{G}$ is flat for any finitely generated projective $R$-module $M$. 

By the characterizations of finite flat modules (\cite[\href{https://stacks.math.columbia.edu/tag/00NX}{Lemma 00NX}]{stacks-project}), it will suffice to express $\big( \Sim_{R}^{k} (M^{\vee} )\big)^{G}$ as a direct summand of a free finite $R$-module. Since the Reynolds operator allows us to write $\big( \Sim_{R}^{k} (M^{\vee} )\big)^{G}$ as a direct summand of $ \Sim_{R}^{k} (M^{\vee})$, we are done if we show that $ \Sim_{R}^{k} (M^{\vee})$ is a free finite $R$-module. But this is trivial since $M$ is a free finite $R$-module. The claim is proved. 

\begin{remark}
Note that \S\ref{subsec:flat} can be deduced from \cite{alper}. However, we have included an alternative proof for the sake of completeness.
\end{remark}

\begin{remark}\label{r:semistable points}
Under the hypothesis of Theorem~\ref{t:Invariant de Sim de Mm es fin pres}, the scheme $\P(M^{\vee})\sslash G:=\Proj \big( \Sim_{R}^\bullet (M^{\vee})\big)^{G}$ is a $R$-scheme of finite presentation and locally projective \cite[Lemma 31.30.4]{stacks-project}.
Consider the tautological line bundle $\mathcal{O}(1)$ on $\P(M^{\vee})$.
The natural inclusion $\big( \Sim_{R}^\bullet (M^{\vee})\big)^{G}\subset \Sim_{R}^\bullet (M^{\vee})$, induces a rational map $\P(M^{\vee})\dashrightarrow \P(M^{\vee})\sslash G$ defined on an open subscheme $U$ of $\P(M^{\vee})$ consisting of those points for which there are sections $s\in\Gamma(\P(M^{\vee}),\mathcal{O}(n))$ (for some $n>0$) not vanishing on them \cite[Corollary 3.7.4]{ega2}. This is the classical definition of semistability, so we will maintain the notation $\P(M^{\vee})^{ss}$ to denote the open subscheme $U$.
\end{remark}

%%%%%%%%%%%%%%%%%%%%%%%%%%%%%%%%%%%%%%%%%%%%%%%%%%%%%%%%%%%%%%%%%%%%
%%%%%%%%%%%%%%%%%%%%%%%%%%%%%%%%%%%%%%%%%%%%%%%%%%%%%%%%%%%%%%%%%%%%
%%%%%%%%%%%%%%%%%%%%%%%%%%%%%%%%%%%%%%%%%%%%%%%%%%%%%%%%%%%%%%%%%%%%

\section{Partially generated graded algebras}\label{sec:pgg}

This section is devoted to study a new class of algebras, which we call \emph{partially generated graded algebras}. These algebras are a generalization of finitely presented graded algebras to the non-noetherian context (for instance, see Proposition~\ref{p:finpres-pgg} and Corollary~\ref{C:noether+pgg imply finite pres}). Among their properties (\S\ref{subsec:projective embd pgg}), we point out  Theorem~\ref{t:projectiveembedding finite degree generated} which will be essential when proving Theorem~\ref{T:closed immersion Proj affine}.

We assume that a base commutative ring $R$ is fixed for this whole section.

%From now on, unless otherwise stated, all rings will be considered to be commutative\todo{Siempre han sido conmutativos. Poner al final de la intro y quitar de aquí.}. We fix a ring $R$ as base ring. 

%%%%%%%%%%%%%%%%%%%%%%%%%%%%%%%%%%%%%%%%%%%%%%%%%%%%%%%%%%%%%%%%%%%%

\subsection{Definition and characterization}\label{subsec:definition pgg}

The notion of $t$-partially generated graded $R$-algebra ($t$-pgg-algebra, for short) is introduced in Definition~\ref{d:pgg-alg}. Then, we show that a graded $R$-algebra is  a $t$-pgg-algebra $A$ if it is generated by elements of degree less than or equal to a given natural number $t$ and  the ideal of relations between the generators satisfies the same condition (Theorem~\ref{t:pgg-exact} and Remark~\ref{R:pgg degree is bounded}). Observe that we do not impose a finiteness condition on the number of generators, but only on the number of different degrees of a generator system of the algebra, as well as on the degrees of a generator system of the ideal of relations. 
%This allows to get analogous results as the classical ones for finitely generated algebras.

\begin{definition}[partial algebra]
Given a subset $I\subseteq \Z$, a partial algebra over $R$ with support $I$ consists of the data $\{A_i, m_{ij}\}$, where each $A_i$ is an $R$-module and 
$m_{ij}:A_i\times A_j\to A_{i+j}$, $i,j\in I$ with $i+j\in I$, are bilinear maps of $R$-modules satisfying:
\begin{enumerate}
\item commutativity: $m_{ij}(a_i,a_j) = m_{j,i}(a_j,a_i)$ for all $a_i\in A_i, a_j\in A_j$ and all pairs $i,j$ with $i,j,i+j\in I$; 
\item associativity: $m_{i+j,k} \circ (m_{ij},1) = m_{i,j+k}\circ (1,m_{jk})$ for all triples $i,j,k$ with $i+j,i+k,j+k ,i+j+k\in I$.
\end{enumerate}
If no confusion arises, we will simply write $A\,=\, \oplus_{i\in I} A_i$.
\end{definition}

\begin{example}
Any (non-graded) $R$-algebra $A$ is also a graded $R$-algebra (setting $A_0:=A$ and $A_i=(0)$ for all $i\neq 0$) and, hence, a partial $R$-algebra. Note that for any $\Z$-graded $R$-algebra, $A=\oplus_{n\in \Z} A_n$, and each subset $I\subset\mathbb{Z}$, it holds that $A_I:=\oplus_{i\in I} A_i$ is a partial algebra. 
\end{example}

\begin{definition}
Given two partial algebras over $R$, $A=\oplus_{i\in I} A_i$ and $B=\oplus_{j\in J} B_j$ with $I,J\subseteq \Z$, a morphism of partial $R$-algebras $f:A\to B$ is a set of morphisms of $R$-modules $\{f_k:A_k\to B_k\,\vert\, k\in I\cap J\}$ such that  $f_{i+j} \circ m_{ij}^A  = m_{ij}^B \circ (f_i, f_j) $ for all $i,j,$ with $i,j,i+j \in I\cap J$.
\end{definition}

%\begin{assumption}
%; for all graded $R$-algebras it will be assumed that $I=\Z_{\geq 0}$ and $A_0=R$; and, for all partial $R$-algebras with $0\in I$, it will be assumed that $A_0=R$. 
% \end{assumption}

Let $\Alg(R)$ denote the category of $R$-algebras, $\gAlg(R)$  the category of $\Z_{\geq 0}$-graded $R$-algebras whose subring consisting of degree zero elements is equal to $R$, and $\pAlg(R)$ the category of partial $R$-algebras whose support $I$ satisfies $0\in I\subseteq \Z_{\geq 0}$ and whose subring consisting of degree zero elements is equal to $R$. 

From now on, the expression $R$-algebra (resp. graded $R$-algebra, partial $R$-algebra) will refer to an object of $\Alg(R)$ (resp. $\gAlg(R)$, $\pAlg(R)$).

It is straightforward to check that the map that sends an object to itself induces fully faithful functors
\begin{equation*}
    \Alg(R) \subset \gAlg(R) \subset \pAlg(R) .
\end{equation*}

% \begin{definition}
For a graded (resp. partial) algebra, $A=\oplus_n A_n$, we say that the elements of $A_n$ are \emph{homogeneous of degree $n$}. We  introduce the following notation
	\begin{equation}\label{eq:degreen}
	\begin{split}
	[A]_n \,&:=\, \{a\in A\, \vert\, \text{$a$ is homogeneous of degree $n$} \}, \\
	[A]_{\leq t} \,&:=\, \oplus_{i=1}^t [A]_i,
	\end{split}
	\end{equation}
and similarly for graded $R$-modules. Further, if no confusion arises, we will simply write  $A_{\leq t}$ instead of  $[A]_{\leq t}$.  
  % \end{definition}

Given $t>0$ and a graded $R$-algebra $A=\oplus_{n\geq 0} A_n$, there is a morphism of functors on $\gAlg(R)$
    \begin{equation}\label{e:Phit}
    	\Phi_{\leq t}: \Hom_{\gAlg(R)}(A,-)  \longrightarrow \Hom_{\pAlg(R)}(A_{\leq t},-)
	 \end{equation}
which, for a graded $R$-algebra $B$, is given by the map
    \begin{equation}\label{e:PhitB}
    \begin{split}
    \Phi_{\leq t}(B): \Hom_{\gAlg(R)}(A,B) & \longrightarrow \Hom_{\pAlg(R)}(A_{\leq t},B).\\
    f \,  & \longmapsto  \Phi_{\leq t}(B)(f):=(f_{i}:=f\vert_{A_{i}} )_{i=0,\hdots,t}
    \end{split}
    \end{equation}
    
\begin{definition}[partially generated graded algebra]\label{d:pgg-alg}
A graded $R$-algebra $A=\oplus_{n\geq 0} A_n$ is called $t$-partially generated graded algebra, for a given natural number $t$ (or, for the sake of brevity, $t$-pgg-algebra), if $A_0=R$ and $\Phi_{\leq t}$ is an isomorphism of functors on $\gAlg(R)$. We simply say that $A$ is a pgg-algebra if there exists a natural number $t$ such that $A$ is $t$-pgg. 
\end{definition}

\begin{remark}\label{R:t-pgg implica t'-pgg}
Observe that if $\Phi_{\leq t_0}$ is an isomorphism, then so is $\Phi_{\leq t}$ for any $t\geq t_0$. Unless explicitly stated, the expression \emph{$t$-pgg} will assume that $t\geq 1$.
\end{remark}

% The next proposition shows that the notion of pgg-algebra generalizes somehow the notion of finitely presented graded algebra in the non noetherian context.

Now, by showing that finitely presented algebras are pgg-algebras, one obtains an important example of this new type of algebras. 

\begin{proposition}\label{p:finpres-pgg}
Let  $R$ be a ring and let $A=\oplus_{n\geq 0} A_n$ be a graded $R$-algebra. 
If $A$ is of finite presentation over $R$, then there exists $t>0$ such that $A$ is a $t$-pgg-$R$-algebra. 
\end{proposition}
 
\begin{proof}
We have to show that there exists $t\gg 0$ such that $\Phi_{\leq t}$ is an isomorphism. First, let us determine $t$. 

Since $A$ is a finitely presented graded $R$-algebra, there exist a finite set of generators $S=\{a_1,\hdots,a_l\}$ such that the kernel of the surjection
	\begin{equation}\label{seq:gen}
	\begin{aligned}
	R[x_1,\hdots,x_l] & \rightarrow A \rightarrow 0 
	\\
	x_i \,& \mapsto a_i
	\end{aligned}
	\end{equation}
is a finitely generated ideal $I_S$. Noting that  $a_i$ can be assumed to be homogeneous, and setting $\operatorname{deg}(x_i):= \operatorname{deg}(a_i)$, the above map is a morphism of graded $R$-algebras whose kernel is a finitely generated homogeneous ideal; let $\{y_1,\ldots, y_m\}$ be a set homogeneous generators of $I_S$. Then, we define:
	\begin{equation}\label{e:syzy}
	\begin{gathered}
	d\,:=\,{\operatorname{max}}\{ \operatorname{deg}(a_i) \,\vert\, i=1,\ldots, l \}, \\
	d^1\,:=\,\operatorname{max}\{\operatorname{deg}(y_i) \,\vert\, i=1,\ldots, m\}, \\
	t\,:=\,  \operatorname{max}\{d,d^1\}.
	\end{gathered}
	\end{equation}

Next, we show that  $\Phi_{\leq t}$ is an isomorphism. Let $B$ be an $R$-algebra. In order to show the injectivity of the map $\Phi_{\leq t}(B)$  of equation~\eqref{e:PhitB}, consider $f,g:A\rightarrow B$ two morphisms of algebras such that $\Phi_{\leq t}(B)(f)=\Phi_{\leq t}(B)(g)$. Since $f_i =g_i$ for each $i=0,\hdots, t$ and $ t \geq d$, we have $f(a_i)=f_{d_i}(a_i)=g_{d_i}(a_i)=g(a_i)$ for each $i=1,\hdots, l$, so $f=g$.

Let us now prove the surjectivity. The set of generators $S$ allows the construction of the map \eqref{seq:gen}. Hence, to give a morphism of algebras $f:A\rightarrow B$ is equivalent to give a morphism of algebras $\widehat{f}:R[x_1,\hdots,x_l]\rightarrow B$ such that $\widehat{f}|_{I_{S}}=0$. Consider now a morphism of $(\phi_{i})\in \Hom_{\pAlg(R)}(A_{\leq t},B)$ and define $b_i:=\phi_{d_i}(a_i)$ for $i=1,\hdots,l$. The elements $b_1,\hdots,b_l$ determine a morphism of $R$-algebras $\widehat{f}:R[x_1,\hdots,x_l]\rightarrow B$ by the rule $x_i \mapsto b_{i}$. It suffices to check that $\widehat{f}|_{I_{S}}=0$ or, what amounts to the same, that $\widehat{f}(y_i)=0$ for all $i=1,\ldots, m$. Since $y_i$ is an algebraic expression in the $x_j$'s, we may write $y_i=\sum \lambda_{r_1,\hdots,r_l}x_1^{r_1}\cdots x_l^{r_{l}}$ and observe that
		\begin{equation*}
		% \begin{split}
		\phi_{d_1}(a_1)^{r_1}\cdots \phi_{d_{l}}(a_l)^{r_{l}} % &
		=\phi_{r_{1}d_{1}}(a_{1}^{r_{1}})\cdots \phi_{r_{l}d_{l}}(a_{l}^{r_{l}})
		% = \\	&
		=\phi_{k}(a_{1}^{r_{1}}\cdots a_{l}^{r_{l}})
		\, .
		% \end{split}
		\end{equation*}
where we denote $k:=\operatorname{deg}(y_i)$ and $k \leq d^1 \leq t$. 

Then,
		\begin{equation*}
		%\begin{split}
		\widehat{f}(y_i) % &
		= \sum \lambda_{r_1,\hdots,r_l}b_1^{r_1}\cdots b_l^{r_{l}}
		%  \\ & 
		= \sum \lambda_{r_1,\hdots,r_l}\phi_{d_1}(a_1)^{r_1}\cdots \phi_{d_{l}}(a_l)^{r_{l}}
		=\phi_{k}(y_i)= 0 \, . 
		%\end{split}
		\end{equation*}
and, being $\phi_{k}$ is a morphism of $R$-modules, it follows that $\widehat{f}$ factorizes through a morphism of $R$-algebras $f:A\to B$ and that $\Phi_{\leq t}(f)=(\phi_i)$.
\end{proof}

Recall from \eqref{eq:degreen} the notation $[-]_n$ for the submodule consisting of homogeneous elements of degree $n$ of a graded algebra (resp. partial algebra or module). Note, for instance, that $[ \oplus_{n\geq  0} A_n ]_i = A_i$. 

Note also that the tensor product of graded objects (rings and modules) inherits a canonical graduation given by setting $\operatorname{deg}(a\otimes b):= \operatorname{deg}(a) + \operatorname{deg}(b)$ for homogeneous elements $a,b$ and, in full generality, 
	\[
	\big[ ( \oplus_{m} A_m ) \otimes ( \oplus_{n} B_n) \big]_i
	\,:=\, \oplus_{j } ( A_j \otimes B_{i-j})
	\]
 This rule can also be used to introduce a graduation in symmetric algebras (of an algebra or a module). In particular, given a graded $R$-module, $( M=\oplus_{n > 0} M_n )$, it holds that $\Sim^{\bullet}_R(M)$,  the symmetric algebra over $R$, is a  graded $R$-algebra whose degree $n$ component is explicitly described as
	\begin{equation}\label{e:Sym-part-n}
	\big[\Sim^{\bullet}_R(M)\big]_n \,=\,
	% \big[ \otimes_{i=1}^t \Sim^{\bullet}_R( M_i)\big]_n \,=\,
	%\bigoplus_{ t > 0} 
	\underset{ \vert \underline{d}\vert =n }\bigoplus  \Sim^{d_1}M_1 \otimes_R \dots \otimes_R  \Sim^{d_n}M_n
	\end{equation}
where $n>0$, $\underline d=(d_1,\dots, d_n)\in (\Z_{\geq 0})^n$. The number $\vert \underline{d}\vert:=\sum_{i=1}^n i \cdot d_i$ will be called the norm of the tuple $\underline d=(d_1,\dots, d_n)\in (\Z_{\geq 0})^n$. We set $\big[\Sim^{\bullet}_R(M)\big]_0= \Sim_{R}^{0}(M):=R$. Note that $\Sim^{\bullet}_R$ defines a functor from the category of graded $R$-modules to $\gAlg$.

Let $( A=\oplus_{n= 0}^t A_n , m_{ij} )$ be a partial $R$-algebra with $A_0=R$. 
It follows from the relation~\eqref{eq:degreen} that 
	\begin{equation}\label{e:Sym-alg-par}
	\Sim^{\bullet}_R(A_{\leq t})  % \,=\, \Sim^{\bullet}_R(\oplus_{n= 0}^t A_n)
	 \,=\, \Sim^{\bullet}_R(\oplus_{n= 1}^t A_n).
	\end{equation}

For a graded $R$-algebra $A=\oplus_{n\geq 0} A_n$ and $t>0$,  the natural inclusion $A_{\leq t} \hookrightarrow A$ together with the product of $A$  induces a canonical homogeneous morphism of graded $R$-algebras
	\begin{equation}\label{e:multi-sym-A}
	\Sim^{\bullet}_R(A_{\leq t}) \to A .
	\end{equation}
Note that $\Sim^{\bullet}_R(A_{\leq t})$ only depends on the additive structure of $A$ as $R$-module but not on its multiplicative structure. The behavior of this map with respect to the multiplicative structure of $A$ is our next task.

\begin{lemma}\label{l:repre}
Let  $R$ be a ring and let $( A=\oplus_{n= 0}^t A_n , m_{ij} )$ be a partial $R$-algebra. On the category $\gAlg(R)$,  it holds that $\Hom_{\pAlg(R)}(A,-)$ is a closed subfunctor of $\Hom_{\gAlg(R)}(\Sim^{\bullet}_R(A_{\leq t}),-)$.
\end{lemma}

\begin{proof}
It suffices to show that there exists an homogeneous ideal $I_t\subseteq \Sim^{\bullet}_R(A_{\leq t})$, depending on $A$ and  $t$,  such that the quotient map $\Sim^{\bullet}_R(A_{\leq t}) \to \Sim^{\bullet}_R(A_{\leq t})/ I_t$ induces an isomorphism
	\begin{equation}\label{e:representable}
	\Hom_{\pAlg(R)}(A,-) \, \simeq 
	\Hom_{\gAlg(R)}(\Sim^{\bullet}_R(A_{\leq t})/ I_t ,-) .
	\end{equation} 
The desired  $I_t$ is given by the homogeneous ideal
	\begin{equation}\label{e:It}
	I_t\,:=\, 
	\big\langle\{ m_{ij}(a_i ,  a_j) - m(a_i \otimes a_j) 
	\,\vert \, a_i \in  A_i, % (\Sim^{\bullet}_R(A_{\leq t}))_i ,
	a_j \in  A_j %  (\Sim^{\bullet}_R(A_{\leq t}))_j 
	\text{ and } 0\leq i,j,i+j\leq t
	\}\big\rangle .
	\end{equation}
Note that the above expression makes sense since $m_{ij}(a_i ,  a_j) \in A_{i+j} \subseteq [\Sim^{\bullet}_R(A_{\leq t})]_{i+j}$ and
	\[
	 m(a_i \otimes a_j) \in \operatorname{Im} \big( 
	 [\Sim^{\bullet}_R(A_{\leq t})]_{i}\otimes [\Sim^{\bullet}_R(A_{\leq t})]_{j} \to [\Sim^{\bullet}_R(A_{\leq t})]_{i+j}
	 \big) . 
	 \]
\end{proof}

\begin{example}\label{e:k[V]}
Let us consider the ring of functions of the vector space $V$, $A:=\Bbbk[V]=\Sim^{\bullet}(V^{\vee})$. The degree $n$ component is  $A_n = \Sim^n (V^{\vee})$. The expression \eqref{e:Sym-part-n} reads
	\[
	\big[\Sim^{\bullet}_R(A_{\leq t})\big]_n \,=\,
	\bigoplus_{\vert \underline{d}\vert =n} \Sim^{d_1}(\Sim^1 V^{\vee})  \otimes_R \dots \otimes_R  \Sim^{d_t}(\Sim^t V^{\vee})
	\]
and, the degree $n$ component of the multiplication map \eqref{e:multi-sym-A} is the sum of the maps
	\[
	\Sim^{d_1}(\Sim^1 V^{\vee})  \otimes_R \dots \otimes_R  \Sim^{d_t}(\Sim^t V^{\vee})
	\,\longrightarrow \, 
	\Sim^n V^{\vee}
	\]
over $\underline{d}$ with $\vert \underline{d}\vert =n$. 
\end{example}

\begin{theorem}\label{t:pgg-exact}
Let  $R$ be a ring and $A$ a graded $R$-algebra  $A=\oplus_{n\geq 0} A_n$. Then, $A$ is a $t$-pgg-algebra if and only if the sequence
	\begin{equation}\label{e:pgg-exact}
	0 \to I_t \to \Sim^{\bullet}_R(A_{\leq t}) \to A \to 0
	\end{equation}
is an exact sequence of graded $R$-modules, where $I_t$ is the homogeneous ideal in~\eqref{e:It}.  In particular, $A$ is a pgg-algebra if and only if \eqref{e:pgg-exact} is exact for some $t > 0$.
\end{theorem}

\begin{proof}
Recall that we say that $A$ is a $t$-pgg-algebra iff $\Phi_{\leq t}$ (see equation~\eqref{e:Phit}) is an isomorphism. 

Given a $t$-pgg-algebra $A$ for some $t>0$, Definition~\ref{d:pgg-alg} and Lemma~\ref{l:repre} show that the following compositon of functors on $\gAlg$
	\[
	\Hom_{\gAlg(R)}(A,-)  
	\overset{\Phi_{\leq t}}\longrightarrow \Hom_{\pAlg(R)}(\oplus_{n=0}^t A_n,-)
	\overset{\sim}\longrightarrow 
	\Hom_{\gAlg(R)}(\Sim^{\bullet}_R(A_{\leq t})/ I_t ,-)
	\]
is an isomorphism. By the uniqueness of the representative of a functor (up to isomorphisms), it holds that $A\simeq \Sim^{\bullet}_R(A_{\leq t})/ I_t$ as graded $R$-algebras. 

For the converse, it suffices to reverse these arguments. 
\end{proof}

\begin{remark}\label{R:pgg degree is bounded}
Note that the exactness of \eqref{e:pgg-exact} can be rephrased as follows: $A$ is a $t$-pgg-algebra (i.e. $\Phi_{\leq t}$ is an isomorphism) if and only if $A$ admits a set of generators of degree less than or equal to $t$ and the ideal of relations among these generators is generated by those relations of degree less than or equal to $t$. Indeed, the latter means that
	\[ % \begin{equation}\label{eq:Itgen}
		I_t\,=\,  [I_t]_{\leq t} \cdot \Sim^{\bullet}_R( [A]_{\leq t})  \,=\, 
		\big( I_t\cap \big[ \Sim^{\bullet}_R( [A]_{\leq t}) \big]_{\leq t} \big)  \Sim^{\bullet}_R( [A]_{\leq t}) .
	\] % \end{equation}
\end{remark}

As a straightforward consequence of Theorem~\ref{t:pgg-exact}, one obtains the following pseudo-converse of Proposition~\ref{p:finpres-pgg}.

\begin{corollary}
    \label{C:noether+pgg imply finite pres}
A noetherian pgg-$R$-algebra is finitely presented over $R$.
\end{corollary}

\begin{proof}
Since $A$ is noetherian, one has that the ideal $\oplus_{t\geq 1} A_t$ is finitely generated. Keeping in mind Proposition~\ref{p:finpres-pgg},  there exists positive integer $n$ and a surjection $p:A^{\oplus n}\to A$. Since $A^{\oplus n}$ is a noetherian $A$-module, $\ker p$ is finitely generated, and the conclusion follows.
\end{proof}

\begin{example}
\label{ex:sato}
Recall from \cite{Sato} that the infinite-dimensional Grassmann manifold of $\Bbbk(\!(z)\!)$, denoted by $\operatorname{Gr}$,  is defined as the subset of points of an infinite dimensional projective space whose homogeneous coordinates satisfy the set of all Pl\"ucker equations. Alternatively, if one introduces  $\operatorname{Gr}$ using scheme theory, then the determinant bundle yields the Pl\"ucker embedding which is a closed immersion into the projective spectrum of the graded algebra $\Bbbk[\{x_S\}]$ where $S$ runs over the set of Young diagrams.  From this point of view, it has been shown in  \cite{Plaza} that the ideal $I\subset \Bbbk[\{x_S\}]$ defining the infinite-dimensional Grassmann manifold is generated by the set of all Pl\"ucker equations, which are quadratic relations in the variables $x_S$. Accordingly, it holds that $\O_{ \operatorname{Gr}} =  \Bbbk[\{x_S\}]/ I $ is a $2$-pgg-algebra. A generalization of this fact to arbitrary $t$-pgg algebras is given in Theorem~\ref{t:projectiveembeddingforpgg}.
\end{example}

%%%%%%%%%%%%%%%%%%%%%%%%%%%%%%%%%%%%%%%%%%%%%%%%%%%%%%%%%%%%%%%%%%%%

%%%%%%%%%%%%%%%%%%%%%%%%%%%%%%%%%%%%%%%%%%%%%%%%%%%%%%%%%%%%%%%%%%%%

\subsection{Properties}
First, we show that being $t$-pgg is local and compatible with direct limits. It must be pointed out that these facts do not hold for the property of being pgg; that is, being pgg is not local and direct limit of pgg may not be pgg. Then, some properties concerned with base change, tensor product, quotients, etc.  are studied.

\begin{proposition}\label{p:t-pgg is local}
Being $t$-pgg is a local property on the base $R$.    
\end{proposition}

\begin{proof}
Note that the sequence \eqref{e:pgg-exact} is exact if and only if for every $x\in \Spec R$ there exists an open affine subscheme $U_x\subseteq \Spec R$ such that \eqref{e:pgg-exact} holds for $A\vert_{U_{x}}$.
\end{proof}

\begin{proposition}\label{p:directlimitoftppgistpgg}
The direct limit of $t$-pgg algebras is again $t$-pgg. 
\end{proposition}

\begin{proof}
Let $B=\varinjlim B_{\lambda}$ where $(B_{\lambda}, u_{\lambda\mu})$ be a direct system of $t$-pgg algebras. The exactness of direct limit and Theorem~\ref{t:pgg-exact} for $B_{\lambda}$ show that the sequence 
    \[
    0 \to \varinjlim   I_{\lambda,t} \to  
    \varinjlim   \Sim_{T}^{\bullet}(B_{\lambda,\leq t}) \to \varinjlim   B_{\lambda} \to  0
    \]
is exact. One concludes by observing that this exact sequence is the sequence~\eqref{e:pgg-exact} corresponding to $B$ since the properties of direct limits, symmetric algebras and truncation imply that
\[
\varinjlim   \Sim_{T}^{\bullet}(B_{\lambda,\leq t}) \simeq 
  \Sim_{T}^{\bullet}(\varinjlim ( B_{\lambda,\leq t}) ) \simeq 
  \Sim_{T}^{\bullet}(\varinjlim  B_{\lambda})_{\leq t} ) \simeq 
\Sim_{T}^{\bullet}(B_{\leq t})
.
\]
\end{proof}

%%%%%%%%%%%%%%%%%%%%%%%%%%%%%%%%%%%%%%%%%%%%%%%%%%%%%%%%%%%%%%%%%%%%

Next, we study some basic operations with pgg-algebras that keep the pgg property, like base changes, tensor products, quotients by ideals. We also state a semicontinuity theorem for finitely presented algebras regarding the natural number $t$ for which the algebra is pgg.

\begin{proposition}
 \label{p:pgg-basechange}
If $A$ is a $t$-pgg algebra over $R$ and $R\to R'$ be a morphism of rings, then  $A_{R'}:=A\otimes_{R} R'$ is a $t$-pgg algebra over $R'$. 
\end{proposition}

\begin{proof}
Tensoring \eqref{e:pgg-exact} by $R'$, we obtain the exact sequence
	\[
	 I_t \otimes_R R' \to \Sim_{R'}^\bullet(A_{\leq t}\otimes_R R') \to A\otimes_R R' \to 0
	\]
Setting $A':=A\otimes_R R'$ and bearing in mind that  
\begin{equation}
\label{e:pgg tensor R prime}
  \Sim^{\bullet}_R(A_{\leq t})\otimes_R R'\simeq \Sim_{R'}^\bullet(A_{\leq t}\otimes_R R') \simeq \Sim_{R'}^\bullet(A'_{\leq t}),  
\end{equation}
the above sequence yields a surjection
	\[
	\xymatrix{
	 I_t \otimes_R R' 
	\ar@{->>}[r]  & K:= \operatorname{Ker}\big(\Sim_{R'}^\bullet(A'_{\leq t}) \to A' \big)
	}\]
which shows that $K$ is generated by elements of degree $t$ at most. The claim is proven. 
\end{proof}

\begin{proposition}\label{p:tensor-pgg}
Let  $R$ be a ring and $A$ (resp. $B$) be a $t_A$-pgg-$R$-algebra such that \eqref{e:Phit} is an isomorphism  for $t_A$ (resp. $t_B$-pgg-$R$-algebra).  Then, it holds that $A\otimes_R B$ is a $t$-pgg-$R$-algebra where $t=t_A + t_B$. %\operatorname{max}
\end{proposition}

\begin{proof}
By Theorem~\ref{t:pgg-exact}, $A$ (resp. $B$) fits into the exact sequence~\eqref{e:pgg-exact}. Let $I_{t_A}^A$ (resp. $I_{t_B}^B$) denote its kernel.  It is straightforward that we have the following commutative diagram of graded $R$-modules
	\[
	\xymatrix@C=18pt{
	& I_{t_A}^A\otimes B  + A\otimes  I_{t_B}^B \ar[r] \ar@{-->}[d] &
	\Sim^\bullet_R(A_{\leq t_A}) \otimes \Sim^\bullet_R(B_{\leq t_B}) \ar[r] \ar@{-->>}[d] &
	A\otimes B \ar[r] \ar@{=}[d] & 0 \\
	0\ar[r] & K \ar[r] & \Sim^\bullet_R \big( [ A\otimes B]_{\leq t}) \ar[r] &
	A\otimes B \ar[r] & 0 }
	\]
where $t:= t_A + t_B$, $K$ is the kernel, the vertical map in the middle is the canonical multiplication induced by $A_{\leq t_A} \otimes B_{\leq t_B}\to [ A\otimes B]_{\leq t}$ and the vertical map on the l.h.s. is the one obtained by factorization. Looking at each degree of the symmetric products of the central terms and having in mind that $A$ (resp. $B$) is generated by homogeneous elements of degree less than or equal to $t_A$ (resp. $t_B$), the surjectivity of the vertical map among symmetric algebras follows. The Snake's Lemma shows that $ I_{t_A}^A\otimes B  + A\otimes  I_{t_B}^B\to K$ is surjective. Noting that the image of 
	\[
	 I_{t_A}^A\otimes B  + A\otimes  I_{t_B}^B \,\to\, \Sim^\bullet_R \big( [A\otimes B]_{\leq t})
	 \]
is generated by homogeneous elements of degree  less than or equal to $t$, the claim is proved. 
\end{proof}

\begin{proposition}\label{p:epi-pgg}
Let  $R$ be a ring and $A$ be a graded $R$-algebra. Let $J\subset A$ be a homogenous ideal generated by its components of degree less than or equal to $t_J$. If $A$ be a $t_A$-pgg-$R$-algebra, then $A/J$ is a $t$-pgg-$R$-algebra  for $t:= \operatorname{max}\{t_A,t_J\}$.  
\end{proposition}

\begin{proof}
Set $t:=\operatorname{max}\{t_A,t_J\}$ and $B:=A/J$.  Using \eqref{e:pgg-exact}, we may build the following commutative diagram of $R$-modules
	\[
	\xymatrix{
	0 \ar[r] &  I_{t}^A \ar[r] \ar@{-->}[d] &  \Sim^{\bullet}_R(A_{\leq t})  \ar[r] \ar@{-->>}[d] & A \ar[r] \ar@{->>}[d] & 0 
	\\
	0 \ar[r] &  K \ar[r]  &  \Sim^{\bullet}_R(B_{\leq t})  \ar[r]^-p & B \ar@{-->}[r] & 0 }
	\]
where $K$ is the kernel of $p:\Sim^{\bullet}_R(B_{\leq t}) \to B$. Since $A\to B$ is a surjective map of graded $R$-modules, it follows that the dashed vertical arrow in the middle is surjective and $p$ is surjective too. 

It remains to check that $K$ is generated by its components of degree less than or equal to $t$. The Snake's Lemma yields
	\[
	I_{t}^A \to K \to L \to 0
	\]
where $L$ is a quotient of $J$. But, noting that $I_{t}^A$ is generated as $\Sim^{\bullet}_R(A_{\leq t})$-module by its components of degree less than or equal to $t$, that $J$ is generated as $A$-module by its components of degree less than or equal to $t$, it then follows that $K$ is generated as $\Sim^{\bullet}_R(B_{\leq t})$-module by  its components of degree less than or equal to $t$. 
\end{proof}

The following result will be particularly useful when we generalize the previous ideas to sheaves.

\begin{proposition}\label{p:point-global-pgg}
Let  $R$ be a ring and $A$ be a graded $R$-algebra of finite presentation. Then,  the function
	\[
	\begin{aligned}
	\widehat{t} \colon S=\Spec R & \longrightarrow \Z \\
	x & \mapsto  \widehat{t}(x):= \operatorname{min}\{ t  \text{ s.t. $\Phi_{\leq t}$ is an isomorphism for $A\otimes_{R} \Bbbk(x)$}\}
	\end{aligned}
	\]
is upper semicontinuous.  
%Further, if $S$ is quasi-compact, ${\mathcal A}$ is a pgg $\O_S$-algebra.
\end{proposition}

\begin{proof}
Let $x\in S$ and set $t:=\widehat{t}(x)$. Consider the exact sequence
	\[
	\ker (\mu) \overset{\iota}\longrightarrow \Sim^{\bullet}_R(A_{\leq t}) \overset{\mu} \longrightarrow A
	\] 
% \todo[inline]{Aquí se dice $\Sim^{\bullet}_R(A_{\leq t})\otimes_{R}\Bbbk(y) = \Sim_{\Bbbk(y)}^{\bullet}( [ A\otimes_{R}\Bbbk(y) ]_{\leq t})$ y no es igualdad a menos que $A_{\leq t}$ sea plano como $R$ módulo.}
For $y\in S$, let $\mu(y)$ denote the map of $\Bbbk(y)$-algebras
	\[
	\Sim^{\bullet}_R(A_{\leq t})\otimes_{R}\Bbbk(y) 
	= \Sim_{\Bbbk(y)}^{\bullet}( [ A\otimes_{R}\Bbbk(y) ]_{\leq t})\,
	\overset{\mu(y)}\longrightarrow \,A\otimes_{R}\Bbbk(y) \,
	\longrightarrow \, 0
	\]
and, for $i\in \Z$, let $[\mu(y)]_i$ be the map of $\Bbbk(y)$-modules obtained by restricting $\mu(y)$ to degree $i$
	\[
	[\mu(y)]_i : [\Sim^{\bullet}_R(A_{\leq t})\otimes_{R}\Bbbk(y) ]_i \, 
	\longrightarrow \,[A\otimes_{R}\Bbbk(y)]_i
	\]
Since $[\mu]_i$ is a morphism between finite type $R$-modules, it follows that 
	\[
	V_i:=\{y \in S \,\vert\,\text{ $[\mu(y)]_i$ is surjective}\}
	\]
is Zariski open in $S$. Since $A$ is generated by elements of degree less than or equal to $t$, it is easy to check that
	\[ U\,:=\, \bigcap_{i=1}^t V_i \,=\, \bigcap_{i\geq 1} V_i \,=\,  \{y \in S \,\vert\, \text{ $\mu(y)$ is surjective}\}\]
and, thus, $U$ an open neighborhood of $x$.  

Next, we will show that there is an open subset $U'\subseteq U$ such that $\widehat{t}(y)\leq t$ for all $y\in U'$. Being $A$ an $R$-algebra of finite presentation, it follows that $[A]_i$ is an $R$-module of  finite presentation  and $[\Sim^{\bullet}_R(A_{\leq t})]_i$ is a  finite type  $R$-module.  Then, by  
\cite[\href{https://stacks.math.columbia.edu/tag/01BP}{Lemma 01BP}]{stacks-project}, $[\ker (\mu)]_i$ is a finite type $R_U$-module. Then
	\[
	U'\,:=\,
	\{ y \in U \,\vert \, \mu_{\leq t}: [\ker (\mu)]_{\leq t} \to [I_t]_{\leq t} \text{ is surjective}\}
	\]
is open and, thus, the function $\widehat{t}$ is upper semi-continuous. 
%Finally, the second claim is straightforward since the hypothesis says that the underlying topological space is quasi-compact. 
\end{proof}

%%%%%%%%%%%%%%%%%%%%%%%%%%%%%%%%%%%%%%%%%%%%%%%%%%%%%%%%%%%%%%%%%%%%
\subsection{Projective immersions}\label{subsec:projective embd pgg}
%%%%%%%%%%%%%%%%%%%%%%%%%%%%%%%%%%%%%%%%%%%%%%%%%%%%%%%%%%%%%%%%%%%%

When a graded $R$-algebra $A=\oplus_{i}A_{i}$ is generated by elements of degree one, one can embed $\Spec(A)$ in the projective bundle $\P(R\oplus A_{1})$ through an open immersion followed by a closed immersion (see \cite{ega2}). We will state an analogous result for $t$-pgg-algebras. Our result is based on the use of the $d$-uple Veronese embedding. 

In what follows,  $\P(M^{\vee})$ will always denote the homogeneous spectrum $\Proj \Sim_{R}^{\bullet}(M)$ where $M$ is regarded as an ungraded $R$-module and, thus, $[\Sim_{R}^{\bullet}(M)]_n = \Sim_{R}^{n}(M)$). Note that we do not require $M$ to be finitely generated.

Recall from \cite[Proposition~2.4.7]{ega2} that for any graded algebra $A=\oplus_{n\geq 0} A_n$ and any $d>0$,  there is a canonical isomorphism 
\begin{equation}
    \label{e:Proj A iso Proj A^d}
    \Proj A \simeq \Proj A^{(d)}
\end{equation}
where
	\[
	A^{(d)} \,:=\, \oplus_{n\geq 0} [ A^{(d)}] _n
	\qquad \text{ with }
	 [ A^{(d)}] _n\,:=\, [A]_{nd}\, .
	 \]

The next result can be seen as a generalization of~\cite[Proposition~3.1.11]{ega2} and~\cite[Lemma in p. 204]{RedBook} where the finiteness hypothesis has been replaced by the $t$-pgg condition. 

\begin{lemma}
Let $A$ be a graded $R$-algebra generated by its elements of degree less than or equal to $t$.  Then, $A^{((t+1)!)}$ is generated by its elements of degree $1$. 
\end{lemma}

\begin{proof}
Fix $n\geq 1$. Consider the commutative diagram of graded $R$-modules
\begin{equation*}
    \xymatrix{
    [ \Sim^{\bullet}(A_{\leq t} )]_{n t!} \ar@{^(->}[r] \ar@{->>}[d] &    \Sim^{\bullet}_R(A_{\leq t})  \ar@{->>}[d]
    \\ 
    A_{n t!} \ar@{^(->}[r]  & A 
    }
\end{equation*}
where the vertical arrow on the right is surjective by hypothesis and thus the map between the terms on the left is also surjective, since they  are  the components of degree $n t!$. There is a corresponding commutative diagram of graded $R$-algebras\begin{equation*}
    \xymatrix{
    \Sim^{\bullet}( [ \Sim^{\bullet}(A_{\leq t} )]_{t!})  \ar@{->}[r]^\phi \ar@{-->}[rd] & 
    (\Sim^{\bullet}(A_{\leq t} ))^{(t!)} \ar@{^(->}[r] \ar@{->>}[d] &    \Sim^{\bullet}_R(A_{\leq t})  \ar@{->>}[d]
    \\ 
    & A^{(t!)} \ar@{^(->}[r]  & A }.
\end{equation*}

For $n\geq t$, the component of degree $n t!$ of $\phi$, $\phi_{n t!}$, is the map
\[
\Sim^{n}( [ \Sim^{\bullet}(A_{\leq t} )]_{t!})  \to 
    [ \Sim^{\bullet}(A_{\leq t} )]_{n t!} 
\]
which, after expansion as in~\eqref{e:Sym-part-n}, reads as follows
\[
\Sim^{n} \big( \underset{ \vert \underline{d}\vert = t! }\oplus  \Sim^{d_1}A_1 \otimes_R \dots \otimes_R  \Sim^{d_t}A_t\big)  \longrightarrow 
    \underset{ \vert \underline{d}\vert = n t! }\oplus  \Sim^{d_1}A_1 \otimes_R \dots \otimes_R  \Sim^{d_t}A_t. 
\]

Now, we use arguments similar to those of the proofs of~\cite[Lemma 2.1.6]{ega2} and~\cite[Lemma in p. 204]{RedBook}. Note that $[ \Sim^{\bullet}(A_{\leq t} )]_{n t!} $ is the linear span of  monomials of the type $\prod_{i=1}^{t} \prod_{j=1}^{d_i} a_{ij}$ where $a_{ij}\in A_{i}$ and $\sum_{i=1}^{t} i d_i=n t!$. Set $\underline{d}=(d_1,\ldots, d_t)$. If $d_i\leq \frac{t!}{i}$, then $\vert \underline{d}\vert = \sum_{i=1}^{t} i d_i <  t \cdot t! \leq n  \cdot t!$. Therefore, if $n\geq t$ and $\vert \underline{d}\vert= n \cdot t!$, then there exists $i_{0}$ such that $d_{i_0}\geq \frac{t!}{i_0}$.  Set $d'_{i}=\frac{t!}{i_0}$ for $i=i_0$, $d'_{i}=0$ for $i\neq i_0$ and $\underline{d'}=(d'_1,\ldots , d'_t)= (0, \ldots,\frac{t!}{i_0},\ldots,0)$. Observe that $\underline{d}'' := \underline{d}-\underline{d}'$ is a vector of non-negative integers. Since the product above may be factored as
    \[
    \prod_{i=1}^{t} \prod_{j=1}^{d_i} a_{ij}
    =
    \prod_{i=1}^{t} \prod_{j=1}^{d'_i} a_{ij}
    \cdot\prod_{i=1}^{t} \prod_{j=1}^{d''_i} a_{i, d'_i+ j}
    \]
where the first term on the right hand side has degree $\vert\underline{d}'\vert= i_{0} d'_{i_0}=t!$ and belongs to $[ \Sim^{\bullet}(A_{\leq t} )]_{t!}$ while the second term has degree $\vert\underline{d}''\vert= (n-1) t!$ and belongs to $[ \Sim^{\bullet}(A_{\leq t} )]_{(n-1)t!}$. Proceeding inductively, we conclude it is the product of $n-t$ terms of degree $t!$, i.e.   lying in $[ \Sim^{\bullet}(A_{\leq t} )]_{t!}$, times a term of degree $t\cdot t!$, i.e.  lying in $[ \Sim^{\bullet}(A_{\leq t} )]_{t\cdot t!}$. 

Accordingly, if $n$ is a multiple of $t$, $\phi_{n t!}$ is surjective and, therefore, the map
\begin{equation}
 \label{e:Att!generateedbydeg1}
 \Sim^{\bullet}( [ \Sim^{\bullet}(A_{\leq t} )]_{t \cdot t!})
\to A^{(t\cdot t!)}
\end{equation}
is surjective. 

Finally,  a careful reading of the previous arguments, one concludes that $t\cdot t!$ in~\eqref{e:Att!generateedbydeg1} can be replaced by $n\cdot m$ where $n$ is an integer with $n\geq t$ and $m$ is a multiple of the l.c.m. of $1, 2, \ldots, t$. In particular, it can be replaced by $(t+1)\cdot t! = (t+1)!$.
\end{proof}

% \todo[inline]{ELIMINAR ESTE TODO??. Las definiciones estándar de cuasi-proyectivo y proyectivo requieren $\P(M^\vee)$ ser de tipo finito. Ahora, si $M$ es arbitrario, es $\varinjlim$ de sus submodulos f.g., por tanto quedaria que lo que estamos considerando son limites proyectivos de morfismos proyectivos. ¿Mencionar al principio?. }

\begin{theorem}\label{t:projectiveembedding finite degree generated}
Let $A$ be a graded $R$-algebra generated by its elements of degree less than or equal to $t$.  Then,  there exists  a canonical closed  embedding of $R$-schemes
	\begin{equation}\label{e:projectiveembeddingforpgg}
	    \Phi:\Proj A \,\hookrightarrow\, 
	{\mathbb P} ( [ \Sim^{\bullet}(A_{\leq t} )]_{(t+1)!}^{\vee} ).
	\end{equation}
	
Moreover, it holds:
\begin{enumerate}
    % \item the closure of the image is isomorphic to $\Proj(A\otimes_R R[T])$.
    \item $\Phi$ is functorial on $R$.
    \item $\Phi$ is equivariant with respect to the natural actions of $\Aut_{\gAlg(R)}(A)$.
    %\todo{revisar si es $\Aut_{\gAlg(R)}(A)$ }
\end{enumerate} 
\end{theorem} 

\begin{proof}
Taking homogeneous spectra in~\eqref{e:Att!generateedbydeg1} and noting~\eqref{e:Proj A iso Proj A^d}
    \[
    \Proj A \simeq \Proj A^{((t+1)!)} \hookrightarrow \Proj \Sim^{\bullet}( [ \Sim^{\bullet}(A_{\leq t} )]_{(t+1)!})
    = \P  ( [ \Sim^{\bullet}(A_{\leq t} )]_{(t+1)!}^{\vee})
    \]
one obtains the closed embedding $\Phi$. Functoriality and equivariance are straightforward.
\end{proof}

The previous result can be adapted to get a locally closed embedding of $\Spec A$ for a pgg algebra $A$. For this goal, let us recall the construction given in~\cite[Chp II, \S8]{ega2}.  Given an arbitrary graded algebra $A=\oplus_{n\geq 0} A_n$, let us consider an indeterminate $T$ with  $\deg T=1$. Then, the morphism of $A$-algebras
	\[
	A[T] \,:=\, A \otimes_R R[T]    \,\longrightarrow\, A
	\]
sending $T$ to $1$ factorizes through the homogeneous localization of $A[T]$ with respect to the multiplicative system generated by $T$, which will be denoted by $A[T]_{\{T\}}^+$; that is,
	\[
	A[T] \,\to  \, A[T]_{\{T\}}^+   \,\overset{\sim}\to\, A
	\]
where the first map is a localization and the second one maps $T\mapsto 1$. The latter is an isomorphism whose inverse sends an homogeneous element $a \in A_n$ to $\frac{a}{T^n}$.  Accordingly, we obtain an open embedding
	\begin{equation}\label{e:spec A into proj A[T]}
	\Spec A \, \hookrightarrow\, \Proj A[T]
	\end{equation}
whose image is dense since its complement is the zero locus of the homogeneous ideal generated by $T$.

Let us now deal with the case of  a graded $R$-algebra, $A$, generated by its elements of degree less than or equal to $t$.  In this case, observe that $R[T]$ with $\operatorname{deg}(T)=1$ is a graded $R$-algebra generated by its elements of degree $1$ and therefore, $\bar A:=A \otimes_R R[T]$ is a graded $R$-algebra generated by its elements of degree less than or equal to $t$.

Applying Theorem~\ref{t:projectiveembedding finite degree generated} to $\bar A$ and recalling~\eqref{e:spec A into proj A[T]}, one has that the composition
    \begin{equation}
        \label{e:specA into proj Sym}
    \Spec A \, \hookrightarrow\, \Proj A[T]
    \, \hookrightarrow\,  
    {\mathbb P} ( [ \Sim^{\bullet}  (\bar A_{\leq t} )]_{(t+1)!}^{\vee}  ))    
    \end{equation}
is a locally closed embedding. 

\begin{proposition}\label{P:loc closed of Spec of pgg}
Let $A$ be a graded $R$-algebra generated by its elements of degree less than or equal to $t$.   Then, there is canonical locally closed embedding of $R$-schemes
\begin{equation}\label{e:spec into proj}
	\Spec A \, \hookrightarrow\, \P  \Big(  \big(
 [ \Sim^{\bullet}(A_{\leq t} )]_{\leq (t+1)!}^{\vee} \Big) 
	\end{equation}
which is functorial in $R$. Furthermore, the closure of the image is isomorphic to $\Proj  A[T]$ where $T$ denotes an indeterminate.
\end{proposition}

\begin{proof}
Bearing~\eqref{e:specA into proj Sym} in mind, we expand $[ \Sim^{\bullet}(\bar A_{\leq t} )]_{(t+1)!}$ as in~\eqref{e:Sym-part-n} and obtain
    \[
    [ \Sim^{\bullet}(\bar A_{\leq t} )]_{(t+1)!}
    = 	 \underset{0\leq  \vert \underline{d}\vert \leq (t+1)! }\oplus  \Sim^{d_1}A_1 \otimes_R \dots \otimes_R  \Sim^{d_{t}}A_{t} \otimes_R \langle T^{(t+1)! - \vert \underline{d}\vert}\rangle .
    \]
Now, sending $T$ to $1$ yields an isomorphism of above identity with 
\[
[ \Sim^{\bullet}(A_{\leq t} )]_{\leq (t+1)!} \simeq \underset{0\leq  \vert \underline{d}\vert \leq (t+1)! }\oplus  \Sim^{d_1}A_1 \otimes_R \dots \otimes_R  \Sim^{d_{t}}A_{t}.
\]
Plugin this into~\eqref{e:specA into proj Sym}, the result follows. 
\end{proof}

It is clear that morphisms~\eqref{e:projectiveembeddingforpgg} and~\eqref{e:spec into proj} are related by the rational morphism induced by the inclusion  $[ \Sim^{\bullet}(A_{\leq t} )]_{ (t+1)!} \to [ \Sim^{\bullet}(A_{\leq t} )]_{ \leq (t+1)!}$. Indeed, there is commutative diagram
\begin{equation*}
    \xymatrix{
    \Spec A \ar@{-->}[d]  \ar@{^(->}[r] &  
    \P  \Big(  \big(
    [ \Sim^{\bullet}(A_{\leq t} )]_{\leq (t+1)!}^{\vee} \Big) \ar@{-->}[d]
    \\ 
    \Proj A  \ar@{^(->}[r]  & 
    \P  \Big(  \big(  [ \Sim^{\bullet}(A_{\leq t} )]_{(t+1)!}^{\vee} \Big)
    }
\end{equation*}

The previous results deal with graded $R$-algebras admitting a generator system consisting of elements whose degree is bounded by $t$. Now, we consider a stronger condition, namely, we discuss the case of a $t$-pgg $R$-algebra.

\begin{theorem}\label{t:projectiveembeddingforpgg}
For a $t$-pgg $R$-algebra $A$, the following properties hold:
\begin{enumerate}
    \item The image of \eqref{e:projectiveembeddingforpgg} is given by intersection of hypersurfaces of degree bounded by $\operatorname{max}\{t,2\}$.
    \item The closure of the image of \eqref{e:spec into proj} is given by intersection of hypersurfaces of degree bounded by $\operatorname{max}\{t+1,2\}$.
\end{enumerate}
\end{theorem} 

\begin{proof}
$(1)$ Let $K$ denote the kernel of  $ \Sim^{\bullet}( [ \Sim^{\bullet}(A_{\leq t} )]_{(t+1)!})
\to A^{((t+1)!)}$. One has to show that the degrees of a generator system of $K$ are bounded by $t$. 

Let us consider the canonical maps of $R$-algebras
\begin{equation*}
    \operatorname{T}^{\bullet} ( [ \Sim^{\bullet}(A_{\leq t} )]_{(t+1)!}) 
    \to 
    \Sim^{\bullet}( [ \Sim^{\bullet}(A_{\leq t} )]_{(t+1)!})
    \to 
      (\Sim^{\bullet}(A_{\leq t} ))^{((t+1)!)}
    \to 
    A^{ ((t+1)!)}
\end{equation*}
where $\operatorname{T}^{\bullet}$ denotes the tensor algebra. Define a ungraded $R$-module  by $B= [ \Sim^{\bullet}(A_{\leq t} )]_{(t+1)!}$. Then, the previous maps yield morphisms of graded $R$-algebras as follows
\begin{equation*}
    \operatorname{T}^{\bullet} B 
    \xrightarrow{f^1} 
    \Sim^{\bullet} B
    \xrightarrow{f^2} 
      (\Sim^{\bullet}(A_{\leq t} ))^{((t+1)!)}
    \xrightarrow{f^3}  
    A^{ ((t+1)!)}.
\end{equation*}

Since $f^1$ is surjective, it follows that 
\[ 
K =\ker(f^3 \circ f^2) = f^1(\ker(f^2\circ f^1)) + (f^2)^{-1}(\ker f^3).
\]
It is easy to check that $A^{ ((t+1)!)}$ is $t$-pgg since $A$ is $t$-pgg (see Remark~\ref{R:pgg degree is bounded}. 
Noting that $\ker(f^2\circ f^1)$ is generated by elements of degree $2$ and that $\ker f^3$ is generated by elements of degree $t$, since $A^{ ((t+1)!)}$ is $t$-pgg, one concludes. 

$(2)$ Recall that $\bar A:=A\otimes_{R} R[T]$ is a $(t+1)$-pgg algebra by Proposition~\ref{p:tensor-pgg} and that the closure of the image of~\eqref{e:spec into proj} is $\Proj \bar A$ by Proposition~\ref{P:loc closed of Spec of pgg}. Applying $(1)$ to $\bar A$, one concludes. 
\end{proof}

%%%%%%%%%%%%%%%%%%%%%%%%%%%%%%%%%%%%%%%%%%%%%%%%%%%%%%%%%%%%%%%%%%%%
%%%%%%%%%%%%%%%%%%%%%%%%%%%%%%%%%%%%%%%%%%%%%%%%%%%%%%%%%%%%%%%%%%%%

\section{Immersions of the subalgebra of invariant elements}\label{sec:algebra of invariants}

%Let  $R$ be a $\Bbbk$-algebra, $M$ a finitely generated projective $R$-module, $T$ a $R$-algebra and $N$ a $T$-module. 

Let us fix the data $(\Bbbk, R, M , T, N, G, A)$ as in Assumptions~\ref{assumptions}.

The goal is to prove Theorem~\ref{T:closed immersion Proj affine}. The main step for this is to show that the subalgebra of invariant elements is a $t$-pgg algebra for certain $t$ (Theorem~\ref{t:algebrainvariantispgg}). Next, we will generalize these results to a global context. Here, the semicontinuity property (Proposition~\ref{p:point-global-pgg}) acquires a special relevance.
We finish discussing how our result is related to some moduli problems in the theory of principal bundles over algebraic curves.

%%%%%%%%%%%%%%%%%%%%%%%%%%%%%%%%%%%%%%%%%%%%%%%%%%%%%%%%%%%%%%%%%%%%
\subsection{Setup}\label{ss:setup}

%{\color{blue} Let $(\Bbbk, R, M , T, N, G, A)$ be as in Assumptions~\ref{assumptions}.}
% Let  $R$ be a $\Bbbk$-algebra, $M$ a finitely generated projective $R$-module, $T$ a $R$-algebra and $N$ a $T$-module. 

\subsubsection{Notation}\label{sss:Notation}

Observe that the functor  $\underline{\Hom}_{\Mod(T)} (  N , M\otimes_{R}  T)$ on $\Alg(T)$ (defined similarly as~\ref{e:functorEnd}) induces a functor on the category of $T$-schemes which is representable by a $T$-scheme. Having in mind the relation
\begin{equation}\label{e:SpecisHom}
\begin{aligned}
     \Big( \Spec \big( \Sim_{T}^\bullet (M^{\vee}\otimes_{R}N) \big) \Big)^{\bullet}(\Spec T') &= \Hom_{\Mod(T)}(  M^{\vee}\otimes_{R}N , T')
    %\Hom_{\Alg(T)}( \Sim_{T}^\bullet(M^{\vee}\otimes_{R}N) , T') 
    \\
     %& \qquad\qquad
   % =     \Hom_{\Mod(T)}(  M^{\vee}\otimes_{R}N , T') 
   & \simeq     \Hom_{\Mod(T')}(  N \otimes_{T} T', M\otimes_{R}  T'), 
        %\\
\end{aligned}
\end{equation}
it is sensible to introduce the following $T$-schemes
\begin{align}
\HOM_{T}(N,M\otimes_{R}T) \,&:=\,
\Spec \big( \Sim_{T}^\bullet (M^{\vee}\otimes_{R}N) \big)
\label{e:Hom como Spec}\\
\P\Hom_{\Mod(T)}(N,M\otimes_{R}T) \,&:=\,
\Proj \big( \Sim_{T}^\bullet (M^{\vee}\otimes_{R}N) \big)
\label{e:Phom as Proj}
\end{align}
% \begin{equation}
%     \label{e:Hom como Spec}
% \HOM_{T}(N,M\otimes_{R}T) \,:=\,
% \Spec \big( \Sim_{T}^\bullet (M^{\vee}\otimes_{R}N) \big)
% \end{equation}
% and
% \begin{equation}    \label{e:Phom as Proj}
% \P\Hom_{\Mod(T)}(N,M\otimes_{R}T) \,:=\,
% \Proj \big( \Sim_{T}^\bullet (M^{\vee}\otimes_{R}N) \big).
% \end{equation}

If $G$ is linearly reductive group $\Bbbk$-scheme, GIT (\cite{mumford-git}) applied to the identity~\eqref{e:Hom como Spec}
yields
\begin{equation}
    \label{e:Hom^G}
\HOM_{T}(N,M\otimes_{R}T) \sslash G  \,=\,
\Spec \big( \Sim_{T}^\bullet (M^{\vee}\otimes_{R}N) \big)^G 
 \,=\,
\Spec A.
\end{equation}
Similarly, in light of Remark~\ref{r:semistable points} (see also GIT, \cite{mumford-git}) we consider the action of $G$ on~\eqref{e:Phom as Proj} and the subscheme of $(\P\Hom_{\Mod(T)}(N,M\otimes_{R}T))$ consisting of semistable points, denoted by $(\P\Hom_{\Mod(T)}(N,M\otimes_{R}T))^{ss}$, and obtain 
\begin{equation}
    \label{e:PHom-ss^G}
(\P\Hom_{\Mod(T)}(N,M\otimes_{R}T))^{ss} \sslash G  \,=\,
\Proj \big( \Sim_{T}^\bullet (M^{\vee}\otimes_{R}N) \big)^G 
 \,=\,
\Proj A.
\end{equation}

\subsubsection{The classical groups}

We recall the definition of some classical groups. To begin with, we summarize the notion of the determinant of a finitely generated projective module $M$ and the group $\SL(M)$, as well as the definition of the group $\operatorname{Sp}(M,q)$  (respectively $\operatorname{O}(M,q), \ \operatorname{SO}(M,q)$) for a given symplectic (respectively, orthogonal) form on $M$.

Let $R$ be a commutative $\Bbbk$-algebra, $M$ a finitely generated projective $R$-module and $f:M\rightarrow M$ an endomorphism. Let $Q$ be another finitely generated projective $R$-module such that $R^{n}=M\oplus Q$. We may extend $f:M\rightarrow M$ to an endomorphism $f\oplus id_{Q}: R^{n}\rightarrow R^{n}$. The determinant $\det(f\oplus id_{Q})$ does not depend on $Q$ \cite[Proposition 1.2]{goldman}, so 
$$
\det(f):=\det(f\oplus id_{Q})
$$
is well defined. Furthermore, $f$ is an isomorphism if and only if $\det(f)\in R^{\times}$ \cite[Proposition 1.3]{goldman} and if $\phi:R\rightarrow T$ is a morphism of rings, then $\det(f\otimes 1)=\phi(\det(f))$, that is, the determinant behaves well under base change \cite[Proposition 1.4]{goldman}. 

A bilinear form on $M$ is a morphism of $R$-modules $b:M\otimes_{R}M\rightarrow R$. It is said to be symmetric (respectively, alternating) if it factorizes (uniquely) by a morphism $b:\Sim^{2}M\rightarrow R$ (respectively, by a morphism $b:\bigwedge^{2}M\rightarrow R$). It is said to be non-degenerate if the induced morphism $\widehat{b}:M\rightarrow M^{\vee}$ is an isomorphism. Finally, a quadratic form is a map of $R$-modules $q:M\rightarrow R$ such that there exists a symmetric bilinear form $b:\Sim^{2}M\rightarrow R$ with $q(m)=b(m\otimes m)$.

A symplectic (respectively, orthogonal) structure on a finitely generated projective $R$-module $M$ is a non-degenerate alternating form $b:\bigwedge^{2}M\rightarrow R$ (respectively, a quadratic form $q:M\rightarrow R$ induced by a non-degenerate symmetric form $b:\Sim^{2}M\rightarrow R$).
We define:
\begin{equation}
\label{E:classical groups}
\begin{split}
\SL(M)&:=\{f\in\Aut_{R}(M)| \ \det(f)=1\},\\
\operatorname{Sp}(M,b)&:=\{f\in\Aut_{R}(M)| \  b=b\circ\wedge^{2}f \},\\
\operatorname{O}(M,q)&:=\{f\in\Aut_{R}(M)| \ q=q\circ f\},\\
\operatorname{SO}(M,q)&:= \operatorname{O}(M,q)\cap \SL(M).
\end{split}
\end{equation}

The symmetric and exterior powers $\Sim^{2}(M)$, $\bigwedge^{2}(M)$ behave well under arbitrary base change (for $\Sim^{2}(M)$, we must impose that $\textrm{char}(R)\neq 2$), so the base change of a symplectic and a orthogonal structure on a $R$-module defines a symplectic and a orthogonal structure on the base change of that module. Thus, we have well-defined group functors $\underline{\SL}(M), \ \underline{\operatorname{Sp}}(M,b), \ \underline{\operatorname{O}}(M,q)$ and $\underline{\operatorname{SO}}(M,q)$. Furthermore, all these group functors are closed subfuntors of $\underline{\Aut}_{\Mod(R)}(M)$, which is representable, so they are representable as well. We will maintain the notation $\SL(M),\operatorname{Sp}(M,b),\operatorname{O}(M,q)$ and $\operatorname{SO}(M,q)$ for their representatives.

Finally, note that the groups in~\ref{E:classical groups} are all linearly reductive when the base ring is noetherian and equi-characteristic of characteristic zero (Proposition~\ref{prop:linearly}).

%%%%%%%%%%%%%%%%%%%%%%%%%%%%%%%%%%%%%%%%%%%%%%%%%%%%%%%%%%%%%%%%%%%%
\subsection{Proof of Theorem~\ref{T:closed immersion Proj affine}}\label{ss:affine case}

This section is devoted to the proof of  Theorem~\ref{T:closed immersion Proj affine}. The main step of our strategy consists of reducing the claim to show that the $T$-algebra $A= \big( \Sim_{T}^\bullet (M^{\vee}\otimes_{R}N) \big)^G$ is a $t$-pgg $T$-algebra. Once this is done,  Theorem~\ref{t:projectiveembedding finite degree generated} and Proposition~\ref{P:loc closed of Spec of pgg} will produce embeddings of~\eqref{e:Hom^G} and~\eqref{e:PHom-ss^G} into suitable projective schemes. The generalization of theses results to the case of non affine schemes will be done in \S\ref{ss:general case}.

It is worth mentioning that, in case $A$ is the algebra of polynomials which are invariant under the natural action of one of the classical groups, we will be able to determine $t$ more precisely thanks to the results of \cite{deConcini-Procesi}.

\begin{proposition}\label{P:FFTSFT-SL}
Let $M$ be a finitely generated projective $\Bbbk$-module and $G$ one of the four groups of~\eqref{E:classical groups}.  Fix a natural number $n>0$ and consider the graded $\Bbbk$-algebra
	\[
	A\,:=\, \Bbbk[(M)^{\oplus n}]^{G} \,=\, \big( \Sim_\Bbbk^\bullet ((M^{\vee})^{\oplus n})\big)^G .
	\]
Suppose that there exists $m$ such that $\rk_{\Bbbk(x)} M_{\Bbbk(x)} \leq m$ for all $x\in \Spec\Bbbk$. Then,  $A$ is a $t$-pgg algebra in the following cases
    \begin{enumerate}
		\item $G=\SL(M)$ and  $t=2 m$;
		\item $2$ is invertible in $\Bbbk$, $G=\operatorname{Sp}(M,b), \operatorname{O}(M,q), \operatorname{SO}(M,q)$ and $t=2 + m$. 
	\end{enumerate}
\end{proposition}

\begin{proof} 
Recall that a finitely generated projective $\Bbbk$-module is a finite rank locally free $\Bbbk$-module. The characterization given in Proposition~\ref{p:t-pgg is local} shows being $t$-pgg is a local property (see also Remark~\ref{R:t-pgg implica t'-pgg}). Hence, it suffices to show the statement for $M$ a free $\Bbbk$-module of rank $m:=\rk M$. The rest of the proof is now a direct consequence of \cite{deConcini-Procesi} which states the First and Second Fundamental Theorems for the groups of~\eqref{E:classical groups} over arbitrary rings.

\emph{Case $G=\SL(M)$}. For  $n< m$, one knows that the subring of invariant elements coincides with the ring of constants, i.e. $A=\Bbbk$, and thus, the statement holds true. So it remains to study the case $n\geq m$. Recalling Theorems~3.3 and~3.4 of~\cite{deConcini-Procesi},  fix a basis $\{e_1,\ldots, e_m\}$ of $M$, then we know that $A$ is generated by the $\binom{n}{m}$ elements of degree $m$ given by
	\begin{equation}
	    \label{e:SL-Plucker}
         [i_{1},...,i_{m}]:M^{\oplus n}\rightarrow\bigwedge^{m}M\simeq\Bbbk
	\end{equation}
%\todo{viene bien tenerlo para luego hablar explícitamente de los invariantes por $\SL$}
where $[i_{1},...,i_{m}](v_{1},...,v_{n}):=v_{i_{1}}\wedge\cdots \wedge v_{i_{m}}$ for $1\leq i_{1}<\cdots <i_{m}\leq n$ and that the relations among them are generated by the Plücker relations, which are quadratic in the generators. Hence $d=m$ and $d^1=2 m$ with the notation of \eqref{e:syzy}, and we conclude arguing as in the proofs of Propositions~\ref{p:finpres-pgg} and~\ref{p:t-pgg is local}. 

\emph{Case $G=\operatorname{Sp}(M,b)$.}
The First Fundamental Theorem  for the symplectic group $\operatorname{Sp}(M,b)$ (\cite[Thm~6.6]{deConcini-Procesi}) asserts that the subring of invariant elements  is generated by the following degree $2$ functions
    \[
    \langle i,j\rangle:M^{\oplus n}\rightarrow\Bbbk\text{,}\hspace{0.5cm}\langle i,j\rangle(v_{1},...,v_{n}):=\eta(v_{i},v_{j})
    \]
for $1\leq i< j\leq n$. Thus,  $A$ is generated by $\binom{n}{m}$ elements of $ \bigwedge^{2}(M^{\oplus n})$; that is, $d=2$.  The  Second Fundamental Theorem for the symplectic group (\cite[Thm~6.7]{deConcini-Procesi}) establishes that if $n\leq m$, the elements $\langle i,j\rangle$ are algebraically independent, and, on the other hand, if $n>m$, the ideal of relations among the generators is generated by the Pfaffians of the  principal minors of size $\frac12 m+1$ of the skew symmetric matrix whose $i, j$-th entry equals $\langle i,j\rangle$.
Therefore, ${d}^{1}=m+2$ and $t=\operatorname{max}\{d,d^1\}=\operatorname{max}\{2,m+2\}=m+2$.

\emph{Case $G=\operatorname{SO}(M,q)$}. The First Fundamental Theorem  for the special orthogonal group (\cite[Thm~5.6]{deConcini-Procesi}) asserts that the subring of invariant elements is generated by the degree $2$ functions
	\[
	\langle i,j\rangle:M^{\oplus n}\rightarrow\Bbbk\text{,}\hspace{0.5cm}\langle i,j\rangle(v_{1},...,v_{n}):=\beta(v_{i},v_{j})
	\]
for $1\leq i< j\leq n$; and by the degree $m:=m$ functions
	\[
	[i_{1},...,i_{m}]:M^{\oplus n}\rightarrow\bigwedge^{m}M\simeq\Bbbk\text{,}\hspace{0.5cm}[i_{1},...,i_{m}](v_{1},...,v_{n}):=v_{i_{1}}\wedge\cdots \wedge v_{i_{m}}
	\]
for $1\leq i_{1}<\cdots <i_{m}\leq n$.

By the Second Fundamental Theorem for $\operatorname{SO}(M,q)$ (\cite[Thm~5.7]{deConcini-Procesi}), the ideal of relations among generators is generated by the $(m+1)\times (m+1)$ minors of the $n\times n$ symmetric matrix whose $i,j$-th entry equals  $\langle i, j\rangle$, and, hence, $d^1=2(m+1) =m+2$ (see \cite[\S5]{deConcini-Procesi}).

\emph{Case $G=\operatorname{O}(M,q)$}.  Continuing with the same notations as in the case of  $\operatorname{SO}(M,q)$, the subring of invariant elements under the natural action of the orthogonal group $\operatorname{O}(M,q)$ is generated by the previously defined functions $\langle i,j \rangle$ for $1\leq i< j\leq n$, and the ideal of relations is generated by the $(m+1)\times (m+1)$ minors of the $n\times n$ symmetric matrix whose $i,j$-th entry equals  $\langle i, j\rangle$. Therefore, $d=2$ and $d^1=2(m+1)$ and $t=2(m+1)=m+2$ (see \cite[\S5]{deConcini-Procesi}).
\end{proof}

\begin{theorem}\label{t:algebrainvariantispgg}
Let $(\Bbbk, R, M , T, N, G, A)$ be as in Assumptions~\ref{assumptions}.
% Let $ R$ be a $\Bbbk$-algebra and $M$ a finitely generated projective $R$-module. Let $G$ be a linearly reductive $\Bbbk$-group scheme together with a representation $\rho:\underline{G}_{R}\rightarrow \underline{\Aut}_{\Mod(R)}(M)$. 

 % Let $T$ be a $R$-algebra, $N$ a $T$-module and consider the graded $T$-algebra:
 %    \begin{equation}
 %        \label{E:A=invariant}
 %     A \,:=\, \big( \Sim_{T}^\bullet (M^{\vee} \otimes_{R}N )\big)^{G}.
 %    \end{equation}

Then, there exists $t$ such that $A$ is a $t$-pgg algebra in the following cases:
\begin{enumerate}
    \item $N$ is a $T$-module of finite presentation, 
    \item  $\Bbbk$ is a equi-characteristic zero noetherian commutative ring, $R=\Bbbk$ and $G$ is one of the four groups in~\eqref{E:classical groups}. In this case $t=2 m$ if $G=\SL(M)$, and $t=2 + m$ for the other cases, where 
    \[ 
    m=\underset{x\in\Spec(\Bbbk)}{\textrm{max}}\{\rk_{\Bbbk(x)}(M_{\Bbbk(x)})\}
    ,\]
    \item  $G$ is a finite group scheme over $\Bbbk$ (in this case, $\vert G\vert$ in invertible in $\Bbbk$ and $t=2 \vert G \vert$).
\end{enumerate}
\end{theorem}

\begin{proof}
\leavevmode\\  \indent
$(1)$ In this case $A$ is of finite presentation by Theorem~\ref{t:main-fin-pres}. We conclude by Proposition~\ref{p:finpres-pgg}.

$(2)$ Under these hypotheses, the groups  in~\eqref{E:classical groups} are linearly reductive. Since every $T$-module $N$ is a direct limit of a direct system of finitely presented $T$-modules $(N_{\lambda},\theta_{\lambda\mu})$ (\cite[Lemma 2]{shannon}), let us write $ N = \varinjlim  N_{\lambda}$. For each $\lambda$,  combining the contragradient representation of $\rho$ (see~\eqref{E:contragradient}) and the trivial representation on $N_{\lambda}$ yields a representation 
\begin{equation}
    \label{E:GTactsonMveeN}
    \rho_{\lambda}:G_{T}^{\bullet}\to \underline{\Aut}_{\Mod(T)}(M^{\vee}\otimes_{R}N_{\lambda})
\end{equation}
defined by $\rho_{\lambda}(g)(\omega\otimes m) := \rho^{\vee}(g)(\omega)\otimes m $. For each $\lambda$, let us consider the subalgebra of invariant elements 
\[
A_{\lambda}\,:=\, \big( \Sim_{T}^\bullet (M^{\vee}\otimes_{R}N_{\lambda})\big)^{G}.
\]
Since $G_{T}$ only acts on $M$ but not on $N_{\lambda}$, it follows that the map $\theta_{\lambda\mu}$ induces a map $u_{\lambda\mu}: A_{\lambda}\rightarrow A_{\mu}$. Thus, we obtain  a direct system $(A_{\lambda},u_{\lambda\mu})$. 

We claim that  
\begin{equation}\label{E:lim-sym-sym-lim-GT}
A \simeq \varinjlim   A_{\lambda} .
\end{equation}
%It is easy to see that $u_{\lambda\mu}: A_{\lambda}\rightarrow A_{\mu}$ is $G_{T}$-equivariant for each pair $\lambda,\mu$. Therefore, the restriction to $(A_{\lambda})^{G}$ takes values in $(A_{\mu})^{G}$. We denote by $v_{\lambda\mu}$ the restriction of $u_{\lambda\mu}$ to the subalgebra of invariant elements. Thus, we have a direct system $((A_{\lambda})^{G},v_{\lambda\mu})$. 
For each $\lambda$, the canonical map $N_{\lambda}\to \varinjlim N_{\lambda} = N$ induces a morphism $v_{\lambda}: A_{\lambda}\rightarrow A$ such that $v_{\lambda} = v_{\mu} \circ u_{\lambda\mu}$ and, therefore, there is a morphism of graded $T$-algebras 
% for each $\lambda$. Clearly, these morphisms are compatible with $v_{\lambda\mu}$. Thus, we have a morphism of $T$-algebras
%and, therefore, a morphism of $R$-algebras 
$$v: \varinjlim    A_{\lambda} \rightarrow A.$$

The composition of the canonical isomorphism of $T$-modules $M^{\vee}\otimes_{R}N\simeq \varinjlim  (M^{\vee}\otimes_{R}N_{\lambda})$, the direct  limit of the Reynolds operators $\varinjlim  (M^{\vee}\otimes_{R}N_{\lambda})\rightarrow \varinjlim  (M^{\vee}\otimes_{R}N_{\lambda})^{G}$, and the  canonical injective morphism $\varinjlim  (M^{\vee}\otimes_{R}N_{\lambda})^{G}\hookrightarrow \varinjlim  (\Sim_{T}^{\bullet}(M^{\vee}\otimes_{R}N_{\lambda}))^{G}$, yields  a morphism of $T$-modules 
\[
M^{\vee}\otimes_{R}N \rightarrow \varinjlim  (\Sim_{T}^{\bullet}(M^{\vee}\otimes_{R}N_{\lambda}))^{G}.
\]
In fact, the later map induces a morphism of $T$-algebras $\Sim_{T}^{\bullet}(M^{\vee}\otimes_{R}N )\rightarrow \varinjlim  (\Sim_{T}^{\bullet}(M^{\vee}\otimes_{R}N_{\lambda}))^{G}$, whose restriction to the subalgebra of invariant elements gives a morphism 
$$v': A \rightarrow \varinjlim   A_{\lambda}.$$ 
It is easy to see that $v$ and $v'$ are inverses to each other and~\eqref{E:lim-sym-sym-lim-GT} follows.

Now, since $N_{\lambda}$ finitely presented, Theorem~\ref{t:main-fin-pres} implies that $A_{\lambda}$ a graded $R$-algebra of finite presentation which, by  Proposition~\ref{p:finpres-pgg}, is $t_{\lambda}$-pgg for certain $t_{\lambda}\in {\mathbb{N}}$. So, if we show that $A_{\lambda}$ is a $t$-pgg algebra with $t=2 m$ if $G=\SL(M)$ and $t=2 + m$ for the other cases, we conclude by Proposition~\ref{p:directlimitoftppgistpgg}. 

Thus, we have to show that $A_{\lambda}$ is a $t$-pgg algebra. Following the argument of~\S\ref{sss:finitepresented} (see, in particular, equations~\eqref{E:casefinpresented1} and~\eqref{E:casefinpresented2}), one obtains that there exist a free finite $T$-module $N'_{\lambda}$, a non-negative integer $k_{\lambda}$ depending on $\lambda$, and an exact sequence of graded $T$-modules
	\[
	\big(  \big( \Sim_{T}^\bullet (M^{\vee} \otimes_R N'_{\lambda})\big)^{\oplus k_{\lambda}}\big)^{G}
	\to 
	B_{\lambda}:= \big( \Sim_{T}^\bullet (M^{\vee} \otimes_R N'_{\lambda})\big)^{G} 	\to A_{\lambda} \to 0 .
	\]
Now, Proposition~\ref{P:FFTSFT-SL} implies that, for all $\lambda$,  the graded $T$-algebra $\big(  \big( \Sim_{T}^\bullet (M^{\vee} \otimes_R N'_{\lambda})\big)^{\oplus k_{\lambda}}\big)^{G}$ is a $t$-pgg algebra where $t=2 m$ if $G=\SL(M)$ and $t=2 + m$ for the other cases.
%(this is the only step in the proof in which we make use of the hypothesis $\Bbbk=R$ is a field). 
In particular, for this choice of $t$, $B$ is a $t$-pgg algebra and the kernel of $B_{\lambda}\to A_{\lambda}$ is generated by elements of degree less than or equal to $t$. Thus, $A_{\lambda}$ is a $t$-pgg algebra by Proposition~\ref{p:epi-pgg}, and the claim is proved.

% \todo[inline]{Me parece que aqui falla algo. Para poder concluir que $t=t_{\lambda}$ tenemos que probar que el ideal de 3.2.2., $\widehat{N}^{G}$, es homogéneo y esta generado por sus componentes de grado "bajo". Ver Prop. 4.28. Asi si podemos concluir por Proposition~\ref{P:FFTSFT-SL}. Otra opcion es considerar $N$ plano. Asi, por el teorema de Lazard, los $N_\lambda$ son libres.} 

$(3)$ The fact that $G$ is finite and linearly reductive implies that $|G|$ is invertible in $\Bbbk$. This case is also called non-modular in the literature. We continue with the notations $N_{\lambda}$, $A_{\lambda}$ of $(2)$. For each $\lambda$, the generalization of the Noether bound to this case (see, among others, \cite{fogarty}) shows that $A_{\lambda}$ is finitely generated over $\Bbbk$ and is generated the elements of degree at most $d: = \vert G\vert$. On the other hand,  the ideal of relations of a minimal set of generators of $A_{\lambda}$ is generated in degree at most  $d^1:= 2 \vert G \vert $ (\cite[Theorem 2]{Derksen-syzy}). Thus, $A_{\lambda}$ is a $d^1$-pgg algebra. Proceeding as above, one concludes that $A=\varinjlim A_{\lambda}$ is a $2 \vert G \vert $-pgg algebra. 
\end{proof}

Combining this result with Theorem~\ref{t:projectiveembedding finite degree generated}, one obtains an embedding of $\Proj A$ (with $A$ as in~\eqref{E:definition A as invariants})  into a ${\mathbb P} ( [ \Sim^{\bullet}(A_{\leq t} )]_{ (t+1)!}^{\vee} )$. Similarly, Proposition~\ref{P:loc closed of Spec of pgg} yields a locally closed immersion of $\Spec A$ into ${\mathbb P} ( [ \Sim^{\bullet}(A_{\leq t} )]_{\leq (t+1)!}^{\vee} )$.

Now we are ready to prove the second main result of the article.

\begin{proof}(of Theorem~\ref{T:closed immersion Proj affine})
All the statements except equivariance with respect to $\underline{\Aut}_{\Mod(T)}(N)$ follow from Theorems~\ref{t:projectiveembedding finite degree generated},~\ref{t:projectiveembeddingforpgg} and~\ref{t:algebrainvariantispgg}.    
Regarding $\underline{\Aut}_{\Mod(T)}(N)$-equivariance, note that
there are canonical actions of $\underline{\Aut}_{\Mod(T)}(N)$ and $\underline{\Aut}_{\Mod(T)}(M^{\vee})$ on $\Sim^{\bullet}_{T}(M^{\vee} \otimes_{R} N)$:
\begin{equation*}
\begin{split}
\alpha:\underline{\Aut}_{\Mod(T)}(N)&\rightarrow \underline{\Aut}_{\gAlg(T)} \big(\Sim^{\bullet}_{T}(M^{\vee} \otimes_{R} N) \big), \ \phi\mapsto S^{\bullet}(1\otimes \phi) \\
\beta:\underline{\Aut}_{\Mod(T)}(M^{\vee})&\rightarrow \underline{\Aut}_{\gAlg(T)} \big(\Sim^{\bullet}_{T}(M^{\vee} \otimes_{R} N) \big),  \ \psi\mapsto S^{\bullet}(\psi\otimes 1) 
\end{split}
\end{equation*}
It is straightforward to show that the two actions commute with each other. Since the action of $\underline{G}_{T}$ on $\Sim^{\bullet}_{T}(M^{\vee} \otimes_{R} N)$ is induced by the composition of the contragredient representation $\rho^{\vee}:\underline{G}_{R}\rightarrow 
\underline{\Aut}_{\Mod(R)}(M^{\vee})$ with $\beta$, it follows that the actions of $\underline{G}_{T}$ and $\underline{\Aut}_{\Mod(T)}(N)$ commute with each other as well. 
Therefore, $\alpha$ yields an action
\begin{equation*}\label{e:AutNinsideAutSMGN}
    \gamma:\underline{\Aut}_{\Mod(T)}(N)\rightarrow \underline{\Aut}_{\gAlg(T)}
    (A) . 
\end{equation*}
%Now, noting that the actions of $\underline{G}_{T}$ and that of $\underline{\Aut}_{\Mod(T)}(N)$ preserve graduations, one obtains  
%\begin{equation*}
%    \delta:\underline{\Aut}_{\Mod(T)}(N)\rightarrow 
%    \underline{\Aut}_{\Mod(T)} ( A_{\leq t})= \underline{\Aut}_{\gAlg(T)} (\Sim^{\bullet}( A_{\leq t})).
%\end{equation*}
%
Finally, $\underline{\Aut}_{\Mod(T)}(N)$-equivariance of $\Phi$ is a consequence of Theorem~\ref{t:projectiveembedding finite degree generated}(2).
% and that the image of this map lies inside 
%\begin{equation*}
%    \underline{\Aut}_{\gAlg(T)} ( [ \Sim^{\bullet}(A_{\leq t} )]_{(t+1)!}^{\vee} ) .
%\end{equation*}
\end{proof}

\begin{remark}\label{R:PGL/G}
The following example illustrates our result in terms of compactifications of homogeneous spaces. Let $(\Bbbk, R, M, T, N, G, A)$ satisfy Assumptions~\ref{assumptions}. Additionally, assume that $\Bbbk$ is a field and $R=T= \Bbbk$, $M=N=\Bbbk^{m}$. Following the discussion in \S\ref{sec:toy model}, we denote 
\[
A=(\Sim^{\bullet}\Hom(\Bbbk^{m},\Bbbk^{m}))^{\SLm}
\]
and let $U$ be the set of semistable points of $\Proj \Sim^{\bullet}\Hom(\Bbbk^{m},\Bbbk^{m})$ under the action of $\SLm$. Thus, we have 
\[
U = \ISOM(\Bbbk^{m}, \Bbbk^{m})/\mathbb{G} = \operatorname{PGL}_{m}
\]
where $\mathbb{G}$ denotes the multiplicative group over $\Bbbk$. Now, let us consider a linearly reductive group $G$ endowed with a faithful representation $\rho: G \hookrightarrow \SLm$. We define
\[
B:=(\Sim^{\bullet}\Hom(\Bbbk^{m},\Bbbk^{m}))^{G}
\]
and let $V$ be the set of semistable points of $\Proj \Sim^{\bullet}\Hom(\Bbbk^{m},\Bbbk^{m})$ under the action of $G$. It is clear that $U\subseteq V$ is an open immersion, and moreover, we have
\[
\operatorname{PGL}_{m} / G = U / G \hookrightarrow V / G = \Proj B \hookrightarrow 
\P \big( [ \Sim^{\bullet}(B_{\leq t} )]_{(t+1)!}^{\vee} \big)
\]
where the first map is an open immersion and the second one (see~\eqref{e:projectiveembeddingforpgg-intro}) is closed ($t$ is such that $B$ is $t$-pgg). Similar arguments allow us to study the $\GLm/G$ case (using Proposition~\ref{P:loc closed of Spec of pgg}).

For references on compactifications of homogeneous spaces, see, for instance, \cite[Corollaire VI.2.6]{Raynaud}, \cite[Theorem 5.2.2]{Brion}. This discussion is also related to results of the type of Chevalley compactification.
\end{remark}

%\todo[inline]{Creo que sería $\GLm/G\subset \Spec(A)\subset 
% {\mathbb P}\big( [ \Sim^{\bullet}(A_{\leq t} )]_{\leq (t+1)!}^{\vee} \big)$ porque al tomar cociente en $Isom\subset Hom$ se mantiene la inclusión (la primera inclusión es abierta y la segunda localmente cerrada). Por otro lado, está claro que $PGL_m\subset \Proj(\Sim^{\bullet}\Hom(\Bbbk^{m},\Bbbk^{m}))^{ss}$ porque $det$ siempre es $G$-invariante si $G\subset \SLm$. De nuevo la inclusión se mantiene al hacer cociente y tendríamos $PGL_m/G\subset \Proj(A)\subset {\mathbb P}\big( [ \Sim^{\bullet}(A_{\leq t} )]_{(t+1)!}^{\vee} \big)$ (la primera inclusión es abierta y la segunda cerrada).}

%%%%%%%%%%%%%%%%%%%%%%%%%%%%%
\subsection{Generalization to non-affine schemes}\label{ss:general case}

Now we discuss how to go from $\Bbbk$-algebras to $\Bbbk$-schemes. Let $X$ be a $\Bbbk$-scheme and $\mathcal{A}$ a sheaf of graded $\O_X$-algebras. We say that $\mathcal{A}$ is $t$-pgg (resp. pgg) for a natural number $t$ if there is a covering of $X$ by affine schemes $\{U_i\}_{i\in I}$ such that $\mathcal{A}(U_i)$ is $t$-pgg (resp. pgg) $\O_X(U_i)$-algebra for all $i$. 

\begin{lemma}\label{L:t-pgg scheaves is local}
    For a sheaf of locally finitely presented graded $\O_X$-algebras, $\mathcal{A}$,  the following conditions are equivalent:
\begin{enumerate}
     \item $\mathcal{A}$ is $t$-pgg.
    \item the stalks $\mathcal{A}_x$ are $t$-pgg $\Bbbk(x)$-algebras for all $x\in X$.
    \item $\mathcal{A}(U)$ is a $t$-pgg $\O_X(U)$-algebra for each affine open subscheme $U\subseteq X$.
\end{enumerate}
\end{lemma}

\begin{proof}
    Proposition~\ref{p:pgg-basechange} shows that $(1)\implies (2)$. Note that $(3)\implies (1)$ is immediate from the definition. Finally, Theorem~\ref{t:pgg-exact} proves both $(2)\implies (1)$ and  $(2)\implies (3)$. 
\end{proof}
\begin{remark}
It is important to note that if $\mathcal{A}$ is a $t$-pgg sheaf of $\O_{X}$-algebras, then the morphism $\Sim^{\bullet}(\mathcal{A}_{\leq t})\to \mathcal{A}$ is surjective and its kernel is quasicoherent.
\end{remark}

There are also more properties relating these notions. For instance, for a sheaf of  locally finitely presented graded  $\O_X$-algebras, $\mathcal{A}$, the set  of points  $x\in X$ such that $\mathcal{A}_{x}$ is $t$-pgg is Zariski open by Proposition~\ref{p:point-global-pgg}. Furthermore, if $X$ is quasicompact, Proposition~\ref{p:finpres-pgg} yields the following implications: \emph{locally of finite presentation $\implies$ $t$-pgg $\iff$ pgg}. However, in general, pgg does not implies $t$-pgg and being pgg is not a local property. % Indeed, Remark~\ref{R:t-pgg implica t'-pgg} implies that $X_{t'}\subseteq X_{t}$ for $t'\geq t$, Proposition~\ref{p:finpres-pgg}  that $X=\cup_{t\in\mathbb{N}} X_{t}$ and  Proposition~\ref{p:point-global-pgg}  that $X_{t}$ is Zariski open in $X$. 
% Finally, combining Proposition~\ref{p:pgg-basechange} and Proposition~\ref{p:point-global-pgg} one shows that $\mathcal{A}(U)$ is $t$-pgg $\O_X(U)$-algebras for all affine open subschemes $U\subseteq X$ if and only if the stalks $\mathcal{A}_x$ are $t$-pgg $\Bbbk(x)$-algebras for all $x\in X$. 
%We conclude that being $t$-pgg is a local property whereas being pgg is not. 

Similarly, the direct limit of sheaves of $t$-pgg-algebras is $t$-pgg again while this may fail for limits of pgg algebras.

\begin{theorem}\label{T:cor-chevalley-seccion}
Let $\Bbbk$ be a commutative ring and $G$  a linearly reductive group scheme over $\Bbbk$. Let $S$ be a $\Bbbk$-scheme and $\mathcal{M}$ a  locally free  $\O_{S}$-module of finite rank
 together with a representation $\rho: \underline{G}_{S} \to \underline\Aut_{\Mod(S)}({\mathcal{M}})$. Let $X$ be $S$-scheme and ${\mathcal{N}}$ be a quasi-coherent $\O_X$-module. 

Assume that one of the following conditions holds
\begin{enumerate}
    \item $X$ quasicompact and $\mathcal{N}$  of finite presentation,
    \item $\Bbbk$ is a equi-characteristic zero noetherian commutative ring, $S=\Spec(\Bbbk)$ and $G$ is one of the four groups in~\eqref{E:classical groups},
    \item $G$ is a finite group scheme over $\Bbbk$ (and, necessarily, $\vert G\vert$ is invertible in $\Bbbk$),
\end{enumerate}
then, there is a natural number $t\geq 2$ such that
\begin{equation}\label{eq:projectiveembeddingforGITquotient-HOM}
    \Phi_{{\mathcal{M}}, {\mathcal{N}},G}:
    \P\Hom_{X}(\mathcal{N},{\mathcal{M}}\otimes_{\mathcal{O}_{S}}\mathcal{O}_{X} ) \sslash 
    G\hookrightarrow {\mathbb P} 
    \big( [ \Sim^{\bullet}(\mathcal{A}_{\leq t} )]_{(t+1)!}^{\vee} \big) ,
\end{equation} 
%$$
%(\P\Hom_{\Mod(T)}(N,M\otimes_{R}T))^{ss} \sslash G
%$$
where $\mathcal{A}_{i}=(\Sim^{i}_{T}(\mathcal{M}^{\vee}\otimes_{\O_{S}}\mathcal{N}))^{G}$, is a  closed embedding of $X$-schemes which is functorial on $X$ and it is $\underline{\Aut}_{graded-\mathcal{O}_{X}-alg}(\mathcal{A})$-equivariant, and thus
$\underline{\Aut}_{\O_X}(\mathcal{N})$-equivariant. Furthermore, $\emph{Im}(\Phi_{\mathcal{M},\mathcal{N},G})$ is given by intersection of degree $t$ hypersurfaces.

In $(2)$ and $(3)$, the number $ t$ depends only on $G$ but not on  $\mathcal{M}$  nor $\mathcal{N}$.
\end{theorem}

\begin{proof}
By Theorem~\ref{T:closed immersion Proj affine}(1) and Lemma~\ref{L:t-pgg scheaves is local}, the whole statement (hypothesis and~\eqref{eq:projectiveembeddingforGITquotient-HOM}) is functorial in both $X$ and $S$, it may be assumed that both $X$ and $S$ are affine schemes. We choose a covering $\{S_i\}_{i\in I}$ by affine open subschemes and, for each $i$, another covering $\{X_{ij}\}_{j\in J_i}$ of $X\times_{S} S_i$ by affine open subschemes. 

For each $i,j$, let $R, T , M , N$ be such that $S_{i}=\Spec R$, $X_{ij}=\Spec T$, $\mathcal{M}\vert_{S_{i}}=\widetilde M$ and $\mathcal{N}\vert_{X_{ij}}=\widetilde N$. Let $A_{ij}$ denote the algebra of invariants corresponding to the restriction to $X_{ij}\to S_i$. By Theorem~\ref{t:algebrainvariantispgg}, $A_{ij}$ is $t_{ij}$-pgg. Quasicompactness of $X$ and Remark~\ref{R:t-pgg implica t'-pgg} imply that we may choose the same $t$ for all $i,j$ in the case $(1)$. Regarding the cases $(2)$ and $(3)$, $t_{ij}$ can be chosen independently of $i,j$ by Theorem~\ref{t:algebrainvariantispgg} again. In all cases, we denote it by $t$.  Theorem~\ref{t:projectiveembedding finite degree generated} yields the desired  closed embedding of $X$-schemes. More precisely, if $\Phi_{ij}$ is the map in~\eqref{eq:projectiveembeddingforGITquotient-HOM} corresponding to the algebra $A_{ij}$, then the maps $\{\Phi_{ij}\}$ glue together to a morphism $\Phi_{{\mathcal{M}}, {\mathcal{N}},G}$ since all objects and constructions are functorial in $X$ and $S$. 

It  remains to check equivariance of $\Phi$. Again, this can be done locally and, thus, we assume that $X$ and $S$ are affine and the result follows as in the proof of Theorem~\ref{T:closed immersion Proj affine}.
\end{proof}

%\begin{corollary}
%Under the hypotheses of Theorem~\ref{th:final},
%%if $G$ is one of the classical (semisimple) linear algebraic $\Bbbk$-groups ($\SLm$, $\emph{Sp}_n$, $\emph{SO}_{m}$), 
%there is a natural locally closed embedding,
%$$\underline{\Hom}_{\Mod(T)}(N,M) \sslash G\hookrightarrow {\mathbb P} ( E_d ),$$
%where $E_d:= \underset{1\leq \vert \underline{d}\vert \leq d }\oplus  \Sim^{d_1}A_1 \otimes_R \dots \otimes_R  \Sim^{d_t}A_t$ and $A_{i}=(\Sim^{i}(M^{\vee}\otimes_{R}N))^{G}$, for a given natural number 
%depending on $G$ but not on $N$ if $G$ is classical.
%\end{corollary}
%\begin{proof}
%The result follows from Theorem~\ref{th:final} and Theorem~\ref{t:projectiveembeddingforpgg}.
%\end{proof}
%Note that the same result holds true in case $N$ is finitely presented and $G$ is an arbitrary semisimple linear algebraic $\Bbbk$-group (follows from Corollary~\ref{t:main-fin-pres}(1), Proposition~\ref{p:finpres-pgg} and Theorem~\ref{t:projectiveembeddingforpgg}) except, possibly, the part regarding the independency of $d\in\mathbb{N}$ on $N$. 

\begin{remark}
It is clear from the proof of Theorem~\ref{P:FFTSFT-SL}, that an analog statement of Theorem~\ref{T:cor-chevalley-seccion} would hold for any pair $(G,M)$, consisting of an algebraic group over a field $\Bbbk$ and a representation on a $\Bbbk$-vector space $M$, as long as there is a FFT and a SFT. Instances of such situations are   $(\End(M))^{\oplus m}$ with the action of  $\GLm$ (\cite{invariant-theory-matrices}) or $M^{\otimes 2r}$ with the action of the orthogonal group  (\cite{SFT-orthogonal}).
\end{remark}

A straightforward adaptation of the proof of Theorem~\ref{T:cor-chevalley-seccion} together with Proposition~\ref{P:loc closed of Spec of pgg} yields the following result. 

\begin{corollary}\label{c:cor-chevalley-seccion}
With the same notations and hypotheses as in Theorem~\ref{T:cor-chevalley-seccion}. 

Assume that one of the following conditions holds
\begin{enumerate}
    \item $X$ quasicompact and $\mathcal{N}$  of finite presentation,
    \item $\Bbbk$ is a equi-characteristic zero noetherian commutative ring, $S=\Spec(\Bbbk)$ and $G$ is one of the four groups in~\eqref{E:classical groups}, $\Spec(\Bbbk)=S$ and $G$ is one of the four groups of~\eqref{E:classical groups},
    \item $G$ is a finite group scheme over $\Bbbk$ (and, necessarily, $\vert G\vert$ is invertible in $\Bbbk$),
\end{enumerate}
then, there is a natural number $t\geq 2$ such that
\begin{equation}\label{eq:locallyclosedembeddingforGITquotient-HOM}
    \HOM_{X}(\mathcal{N},{\mathcal{M}}\otimes_{\mathcal{O}_{S}}\mathcal{O}_{X} ) \sslash 
    G\hookrightarrow {\mathbb P} 
    \big( [ \Sim^{\bullet}(\mathcal{A}_{\leq t} )]_{\leq (t+1)!}^{\vee} \big) ,
\end{equation} 
where $\mathcal{A}_{i}=(\Sim^{i}_{T}(\mathcal{M}^{\vee}\otimes_{\O_{S}}\mathcal{N}))^{G}$, is a locally closed embedding of $X$-schemes which is functorial on $X$ as $X$ varies in the category of $S$-schemes.

In $(2)$ and $(3)$, the number $d $ depends only on $G$ but not on  $\mathcal{M}$  nor $\mathcal{N}$.
\end{corollary}

\begin{remark}\label{R:applications moduli Schmitt}
It is remarkable that particular cases of the closed embedding~\eqref{eq:projectiveembeddingforGITquotient-HOM} have already appeared in the literature when dealing with the study of moduli spaces of principal $G$-bundles (see \cite{GomezSolsSchmitt}, \cite[\S3]{schmitt}, \cite{schmitt-book}, \cite{munoz1}).
%\todo[inline]{explicitar aquí las expresiones explícitas de cocientes de álgebras simétricas de Schmitt. Muy breve, unas líneas.}
To be more precise, if a finite-dimensional representation $\rho:G\hookrightarrow \SL(V)$ is given, principal $G$-bundles over a projective algebraic curve defined over an algebraically closed field $\Bbbk$ of characteristic zero can be seen as pairs $(\mathcal{E},\tau)$, $\mathcal{E}$ being a locally free sheaf of rank equal to $\operatorname{dim}_{\Bbbk}(V)$ and $\tau$ a global section of the structure morphism 
\[
\ISOM(\mathcal{N},{\mathcal{M}}\otimes_{\Bbbk}\mathcal{O}_{X} )/G \rightarrow X
\] 
where $\mathcal{N}=\mathcal{E}$ and $\mathcal{M}=V^{\vee}$. 

If $G$ is equal to $\SL(V), \operatorname{Sp}(V,q)$ or $\operatorname{SO}(V,q)$, and the curve is smooth, the schemes $\ISOM(\mathcal{N},{\mathcal{M}}\otimes_{\Bbbk}\mathcal{O}_{X} )/G$ admit an explicit description (see \cite[\S 2.1]{schmitt-book}).
Since $
\ISOM(\mathcal{N},{\mathcal{M}}\otimes_{\Bbbk}\mathcal{O}_{X} )/G\subset \HOM_{X}(\mathcal{N},{\mathcal{M}}\otimes_{\Bbbk}\mathcal{O}_{X} ) \sslash G$, the $\underline{\Aut}_{\O_X}(\mathcal{N})$-equivariant embedding of $\P\Hom_{X}(\mathcal{N},{\mathcal{M}}\otimes_{\Bbbk}\mathcal{O}_{X} ) \sslash G$ in $\P([\Sim^{\bullet}(\mathcal{A}_{\leq t})]^{\vee}_{t!})$ given by~\eqref{eq:projectiveembeddingforGITquotient-HOM} allows, for example, to define a semistability condition for the pair $(\mathcal{E},\tau)$, which is  a crucial point in the construction of these moduli spaces (\cite[Lemma 2.4.2.7]{schmitt-book}), even when working over singular curves.
\end{remark}

\subsection{Future applications to moduli problems}\label{ss:future}

Let us conclude this section by explaining how our results can be applied to the study of certain moduli problems. Indeed, the moral of our paper is that the constructions of \cite{schmitt} are canonical and can be applied in many different settings. 

Now, we may describe the important consequences that Theorem~\ref{t:Invariant de Sim de Mm es fin pres} has  in certain moduli problems. Let $\mathcal{U}$ be a $\Bbbk$-scheme (finite-dimensional or not), $X$ a projective curve over defiend over a field $\Bbbk$, $\mathcal{E}$ a locally free sheaf of rank $n$ on $\mathcal{U}_{X}:=\mathcal{U}\times X$, $G$ a semisimple linear algebraic $\Bbbk$-group $G$, and $\rho: G\hookrightarrow \SL(V)$ a faithful representation of dimension $n$, i.e. $V:=\Bbbk^n$. Then, Theorem~\ref{t:Invariant de Sim de Mm es fin pres} says that the subalgebra of invariant elements $(\Sim^{\bullet}_{\O_{\mathcal{U}_X}}(\mathcal{M}^{\vee}\otimes_{\O_{\mathcal{U}_X}}\mathcal{E}))^{G}$ (where $\mathcal{M}:=V^{\vee}\otimes_{\Bbbk}\O_{\mathcal{U}_X}$) is locally of finite presentation. Note that this is important for ensuring the existence of the scheme of sections, $\operatorname{Sec}\rightarrow \mathcal{U}$, of 
$\HOM_{\mathcal{U}_{X}}(\mathcal{E}, \mathcal{M})\sslash G \rightarrow \mathcal{U}_{X}$
\cite{ry} in case $\mathcal{U}$ is not noetherian. If $\mathcal{U}$ is the moduli space of locally free sheaves on $X$ with formal trivializations (see \cite{alvarez,mulase,plaza-kp}) and $\mathcal{E}$ is the universal locally free sheaf, $\operatorname{Sec}\rightarrow \mathcal{U}$ is  closely related  with the moduli space of principal $G$-bundles on $X$ with formal trivializations (see \cite{bhosle,langer,munoz1,schmitt} for the formalism of singular principal bundles). Although the functor defined by pairs given by a principal $G$-bundle and a formal trivialization is known to be representable by an infinite-dimensional scheme \cite[Proposition 4.1.5]{Frenkel},  an approach based in our results would lead to a global description of it. Indeed, such description would allow to determine a system of infinitely many differential equations describing the moduli space inside certain infinite-dimensional Grassmannian (see \cite{curves-eqs, hurwitz-eqs, higgs-eqs} for related works). We plan to work this problem out in a forthcoming paper.

Finally, Corollary~\ref{c:cor-chevalley-seccion} may be useful to study lifting properties of principal $G$-bundles over curves (see \cite{Rama} for a detailed exposition when the curve is the projective line) if a faithful representation $\rho:G\hookrightarrow \SL(V)$ is given. More precisely, let $S=\Spec(R)$ be the spectrum of a $\Bbbk$-discrete valuation ring, $\pi:\mathcal{X}\rightarrow S$ a flat family of (polarized) stable curves of genus $g$, the generic fiber $\mathcal{X}_{\eta}$ being smooth, and $\operatorname{SPB}(\rho)^{ss}$ the moduli space of semistable singular principal $G$-bundles on $\mathcal{X}$ \cite{munoz1}. One can ask for those singular principal $G$-bundles $(\mathcal{F}_0,\tau_0)\in \operatorname{SPB}(\rho)^{ss}_{0}$ (the special fiber of $\operatorname{SPB}(\rho)^{ss}$) that arise as deformations of semistable principal $G$-bundles on $\mathcal{X}_{\eta}$. Since $\tau_0$ belongs to the scheme of sections, $\operatorname{Sec}_0$, of  $\HOM_{\mathcal{X}_0}(\mathcal{F}_0, \mathcal{M}_0) \sslash G \rightarrow \mathcal{X}_0$ (here $\mathcal{M}:=V^{\vee}\otimes_{\Bbbk}\O_{\mathcal{X}_0}$), the locally closed embedding~\ref{eq:locallyclosedembeddingforGITquotient-HOM} allows to translate the lifting problem of pairs $(\mathcal{F}_0,\tau_0)$ into a lifting problem of sections of certain projective bundle. 
%Note that, in this situation, the embedding~\eqref{eq:locallyclosedembedding-intro}\todo{check} is globally well-defined. 
Understanding these lifting properties would allow to understand some geometric properties of certain compactifications of the universal moduli space of principal $G$-bundles over $\overline{\mathcal{M}_g} $\cite{munoz2,munoz3}.

%%%%%%%%%%%%%%%%%%%%%%%%%%%%%%%%%%%%%%%%%%%%%%%%%%%%%%%%%%%%%%%%%%%%
%%%%%%%%%%%%%%%%%%%%%%%%%%%%%%%%%%%%%%%%%%%%%%%%%%%%%%%%%%%%%%%%%%%%
%%%%%%%%%%%%%%%%%%%%%%%%%%%%%%%%%%%%%%%%%%%%%%%%%%%%%%%%%%%%%%%%%%%%

\bibliographystyle{amsplain}
\bibliography{mibiblio2}

\end{document}